\documentclass[12pt]{article}
\usepackage[english]{babel}
\usepackage{graphicx}
\usepackage{framed}
\usepackage[normalem]{ulem}
\usepackage{skak} 
\usepackage{amsthm}
\usepackage{amssymb}
\usepackage{amsfonts}
\usepackage{enumitem}
\usepackage[utf8]{inputenc}
\usepackage{dsfont} 
\usepackage{mathtools}
\usepackage{hyperref}
\usepackage[top=1 in,bottom=1in, left=1 in, right=1 in]{geometry}
\usepackage[toc,page]{appendix}
\usepackage{todonotes}
\usepackage{bm}
\usepackage{amsmath}
\usepackage{cleveref}
\usepackage{longtable}
\usepackage{array}

\DeclareMathAlphabet\mathbfcal{OMS}{cmsy}{b}{n}

\newcommand\funcRestr[2]{{
		\left.\kern-\nulldelimiterspace 
		#1 
		\vphantom{\big|} 
		\right|_{#2} 
}}

\theoremstyle{plain}
\newtheorem{theorem}{Theorem}[section]
\newtheorem{lemma}[theorem]{Lemma}
\newtheorem{definition}[theorem]{Definition}
\newtheorem{corollary}[theorem]{Corollary}

\newtheorem{assumption}[theorem]{Assumption}
\newtheorem{remark}[theorem]{Remark}

\newcommand{\measPPPSection}{{\lambda}}
\newcommand{\mollifier}{\xi^{\mathrm{space}}}
\newcommand{\mollifierTime}{\xi^{\mathrm{time}}}
\newcommand{\transportPlan}{\gamma}

\newcommand{\notinclude}[1]{}
\newcommand{\weakstarto}{{\stackrel*\rightharpoonup}}
\newcommand{\dist}{{\mathrm{dist}}}
\newcommand{\mdiv}{{\rm div}}
\newcommand{\R}{\mathbb{R}}
\newcommand{\N}{\mathbb{N}}

\newcommand{\Z}{\mathbb{Z}}
\newcommand{\C}{\mathbb{C}}
\renewcommand{\emptyset}{\varnothing}

\renewcommand{\d}{{\mathrm d}}

\newcommand{\norm}[1]{\left\lVert#1\right\rVert}
\newcommand{\normS}[1]{\lVert#1\rVert}
\newcommand{\absNorm}[1]{\left\lvert#1\right\rvert}
\newcommand{\absNormS}[1]{\lvert#1\rvert}
\newcommand{\Log}[1]{\log\left(#1\right)}

\newcommand{\DetectorPair}[3][]{\Gamma_{#1}^{#2}\times\Gamma_{#1}^{#3}}

\newcommand{\HausdorffMeas}[1]{\HausdorffMeasSq\left( #1 \right)}
\newcommand{\HausdorffMeasSq}{\hd^2\otimes\hd^2}
\newcommand{\dotProd}[2]{\langle #1,#2 \rangle}
\newcommand{\restr}{\with}
\newcommand{\convN}{\xrightarrow{n\rightarrow\infty}}
\newcommand{\convStar}{\xrightharpoonup{\ast}}

\newcommand{\Prob}[1]{\mathbb{P}\left( #1 \right)}
\newcommand{\Expect}[1]{\mathbb{E}\left[ #1 \right]}
\newcommand{\vol}{{\mathrm{vol}}}
\newcommand{\volS}{\mathrm{vol}_{S^2}}
\newcommand{\volG}{\volS}

\newcommand{\restrictTo}[1]{|_{#1}}

\newcommand{\dint}{\,\d}

\DeclarePairedDelimiter\ceil{\lceil}{\rceil}

\newcommand{\BBEnergy}{S}

\newcommand{\density}{\rho}
\newcommand{\densityT}[1]{\density_{#1}}
\newcommand{\densityN}[1]{\density_{#1}}
\newcommand{\densityTN}[2]{\density_{#2, #1}}

\newcommand{\momentum}{\eta}
\newcommand{\momentumN}[1]{\eta_{#1}}

\newcommand{\momentumTN}[2]{\eta_{#2, #1}}

\newcommand{\forwardOp}{A}
\newcommand{\UF}{u}
\newcommand{\forwardOpP}{\forwardP}
\newcommand{\forwardOpD}{A^d}

\newcommand{\forwardOpPN}{A^{\UF_n}}

\newcommand{\forward}[1]{A#1}

\newcommand{\forwardP}[1]{A^{\UF}#1}
\newcommand{\forwardD}[1]{A^{d}#1}
\newcommand{\forwardS}[1]{A^{s}#1}

\newcommand{\forwardPN}[2]{A^{\UF_{#2}}#1}

\newcommand{\forwardBinnedOp}{B}

\newcommand{\forwardBinnedP}[1]{B^\UF#1}

\newcommand{\forwardBinnedPN}[2]{B^{\UF_#2}_{#2}#1}

\newcommand{\forwardBinnedLimit}[1]{B^\UF_{\infty}#1}
\newcommand{\halflife}{{T_{1/2}}}
\newcommand{\dom}{D}
\newcommand{\domDelta}{\mathcal{D}}
\newcommand{\domDeltaHalf}{D_{\delta/2}}
\newcommand{\boundDomDeltaSq}{{\partial\domDelta\times\partial\domDelta}}

\newcommand{\spacetime}{{[0,T]\times\dom}}
\newcommand{\spacetimedet}{{[0,T]\times\boundDomDeltaSq}}

\newcommand{\intensity}{q}

\newcommand{\timeInterval}{[0,T]}

\newcommand{\meas}{{\mathcal M}}
\newcommand{\measp}{{\mathcal M_+}}

\newcommand{\pushforward}[2]{{{#1}_{\#}#2}}
\newcommand{\hd}{\mathcal{H}}
\newcommand{\hdhd}{\HausdorffMeasSq}

\newcommand{\probSpace}{\Omega}
\newcommand{\sigmaProb}{\mathcal{F}}
\newcommand{\Poi}[1]{{\mathbfcal{P}(#1)}}

\newcommand{\att}{{\mathrm a}}
\newcommand{\sct}{{\mathrm s}}
\newcommand{\dt}{{\mathrm d}}

\newcommand{\RNderivative}[2]{{\frac{\d{#1}}{\d{#2}}}}

\begin{document}
\newpage
\title{Convergence of Poisson point processes and of optimal transport regularization with application in variational analysis of PET reconstruction}
\author{Marco Mauritz \and Benedikt Wirth}
\maketitle

\begin{abstract}
  Poisson distributed measurements in inverse problems often stem from Poisson point processes that are observed through discretized or finite-resolution detectors,
  one of the most prominent examples being positron emission tomography (PET).
  These inverse problems are typically reconstructed via Bayesian methods.
  A natural question then is whether and how the reconstruction converges as the signal-to-noise ratio tends to infinity
  and how this convergence interacts with other parameters such as the detector size.
  In this article we carry out a corresponding variational analysis for the exemplary Bayesian reconstruction functional from \cite{PET_Base, Schmitzer_DynamicCellImaging},
  which considers dynamic PET imaging (i.e.\ the object to be reconstructed changes over time) and uses an optimal transport regularization.
\end{abstract}

\section{Introduction}
Inverse problems with Poisson distributed measurements collected from finitely many detectors are often reconstructed using Bayesian methods.
One of the most prominent examples is positron emission tomography (PET) imaging.

In PET one tries to reconstruct a radionuclide distribution within an object (e.g.\ a patient or a lab animal).
A radioactive decay produces two photons to be emitted in opposite directions from (more or less) the decay position.
Detectors around the object recognize the two simultaneous photons (a so-called coincidence) and thus the line segment along which the decay happened.
The Poisson distribution of these measurements derives from the Poisson distribution of radioactive decay.

In static PET imaging, however, the Poisson noise is often negligible in practice: By increasing the imaging time interval it can readily be reduced.
This is not the case in dynamic PET imaging, where the radionuclide distribution changes over time
and which is therefore much more interesting and complicated.
There are several ways to overcome the difficulties of dynamic inverse problems,
and recently regularization of dynamically changing measures via optimal transport became more prevalent \cite{Schmitzer_DynamicCellImaging, Bredies_Fanzon_dynamicIP, Bredies_Carioni_CondintionalGradientMethodOTRegularization}.

A standard exercise in inverse problems is to prove convergence of the reconstruction in the limit of vanishing noise or rather of infinite signal-to-noise ratio (SNR).
If at the same time the measurement resolution increases and other regularizing parameters (like the regularization weight) decrease, one can hope to converge to the ground truth.
A priori, however, it is not obvious how the resolution and other parameters need to be coupled to the SNR
-- typically in inverse problems the regularization is not allowed to decrease too fast in comparison to the noise.
In the Poisson noise setting this is challenging since the noise is not independent of the ground truth.

In this article we prove convergence of the reconstruction to the ground truth for an exemplary model of dynamic PET imaging \cite{PET_Base, Schmitzer_DynamicCellImaging},
which is particularly interesting due to several involved factors:
Poisson noise, temporal dependence, measure-valued reconstructions, and optimal transport regularization.
Essentially, we prove $\Gamma$-convergence of the reconstruction functional.
Though this convergence or stability result is the first step in analysing such an inverse problem, it is already quite nontrivial.
The next step would concern convergence rates under source conditions,
where in the case of measure-valued reconstructions already the metric to be employed is unclear (potential metrics could be borrowed from the literature on superresolution \cite{Candes-Fernandez-Granda,Peyre-Duval-Denoyelle,Holler_Schlueter_Wirth}).

In the considered dynamic PET reconstruction method from \cite{PET_Base, Schmitzer_DynamicCellImaging},
the ill-posed inverse problem is regularized by means of optimal transport.
This approach guarantees temporal consistency between different measurement times and favours temporal evolutions with low kinetic energy.
In more detail, a maximum a posteriori (MAP) estimate leads to minimizing the functional 
\begin{equation*}
J^{E_q}(\density,\momentum) = \norm{\forwardOp\density}-\frac{1}{q}\int_{[0,T]\times\boundDomDeltaSq}\log\left( \forwardBinnedP{\density} \right)\dint E_q+\beta\BBEnergy(\density,\momentum).
\end{equation*}
Here, $\norm{\cdot}$ is the total variation norm on the space of measures, i.e.\ the total mass of a nonnegative
measure, $\domDelta\subset\R^3$ is the PET scanner interior on whose boundary $\partial\domDelta$ the detectors are located, and $E_q = \sum_{k=1}^{K_q}\delta_{(t_k, a_k,b_k)}$ is the PET measurement of all coincidences in so-called \textit{listmode format}, which is represented as a linear combination of Dirac measures at time points $t_k$ and detector pairs $(a_k,b_k)\subset\partial\domDelta\times\partial\domDelta$.
The linear forward operator that maps a radioactive radionuclide distribution $\density$ to an expected distribution of coincidences on $\R\times\boundDomDeltaSq$ is denoted $A$, and $B^u$ is a modification that accounts for the discrete nature of the measurements (it depends on the detectors $\Gamma^j\subset \partial\domDelta$ and the temporal resolution) and helps to reduce a certain bias of the MAP estimate via the parameter $\UF$.
The auxiliary variable $\momentum$ is an $\R^3$-valued Radon measure representing the physical momentum associated with the motion of the mass $\density$, thus both variables must satisfy the \textit{continuity equation}
\begin{align}\label{eqn:CE_introduction}
\partial_t\density+\mdiv\,\momentum=0.
\end{align}
Finally, $q>0$ is a scaling factor proportional to the expected number of events (meaning that on average we have $\normS{E_q}\simeq q$), the parameter $\beta>0$ is a regularization weight, and $\BBEnergy$ is the so-called Benamou--Brenier functional (a dynamic formulation of the Wasserstein-2 optimal transport cost),
\begin{equation}\label{eqn:BenamouBrenier}
\BBEnergy(\density,\momentum)=\begin{cases}
\int_0^T\int_\domDelta\left( \frac{\dt\momentum_t}{\dt\densityT{t}} \right)^2\dt\densityT{t}\ \dt t \quad&\text{if }\density\ge 0\text{ and \eqref{eqn:CE_introduction} holds},\\
\infty&\text{else}.
\end{cases}
\end{equation}
In this article we use $\Gamma$-convergence to investigate the limit behaviour of the PET model for a SNR tending to infinity (corresponding to $q\to\infty$ and simultaneous weak-* convergence of $E_q/q$). In general, a higher SNR should lead to better reconstructions which is indeed the case for the above model as will be seen from the $\Gamma$-limit.

For an increasing SNR it makes sense to also vary other system parameters (e.g.\ to simultaneously increase the detector resolution). In our variational analysis we also cover the situation in which the detector sizes and smallest resolved time difference approach zero, the unbiasing factor $\UF$ may converge to any positive and the regularization weight $\beta$ to any nonnegative number.
In the end this allows to prove stability of the reconstruction and reconstruction of the ground truth in the vanishing noise limit.

Since PET measurements result from radioactive decay, they are of stochastic nature and follow a Poisson distribution. This stochastic behaviour is incorporated into our analysis, and we use Poisson point processes (PPP) to describe radioactive decay. Likewise, the PET measurements are described by a PPP. The growing SNR is modelled by an increasing intensity of the PPP (corresponding e.g.\ to a decreasing halflife of the considered radionuclide). Therefore, we need to understand the convergence of Poisson point processes to be able to compute the $\Gamma$-limit. More precisely: the measurements are realizations of a PPP $\bm{E}_q$ with intensity measure $q\forwardOp\density^\dagger$ for a finite ground truth radionuclide distribution $\density^\dagger$. With $q\to\infty$ the average number of points being sampled from the measurement process $\bm E_q$ also tends to infinity and we study the convergence properties of $\frac{1}{q}\bm E_q$.

\subsection{Contributions of the article}
Our main contributions and the outline of this article are as follows:
\begin{itemize}
\item In \cref{SectionPPP} we prove, based on \cite{PPPconcentrationInequality}, convergence results of PPP that are important for the $\Gamma$-convergence, but also interesting in their own right. For a finite measure $\lambda$ on some measurable space $X$ and a monotone sequence $q_n\to\infty$ we consider the PPP $\bm E_{q_n}$ with intensity measure $q_n\lambda$. For a sequence of partitions $(C_n^k)_{k=1,\ldots,N_n}$ of $X$ we show
\begin{align*}
    \frac{1}{r_n}\sum_{k=1}^{N_n}\absNorm{\frac{1}{q_n}\bm E_{q_n}(C_n^k)-\lambda(C_n^k)}\xrightarrow[n\rightarrow\infty]{\text{a.s.}}0
\end{align*}
under suitable conditions on $q_n$, $N_n$, and $r_n$.
We distinguish between two different ways of defining the sequence $\bm E_{q_n}$, modelling different experimental settings:
an arbitrary sequence, corresponding to potentially independent measurements (of the same fixed ground truth),
and a coupled sequence, in which previous data is augmented by new measurements.
The latter results in slightly less restrictive conditions for convergence.

\item In \cref{sec:BenamouBrenier} we approximate measures $\density=\d t\otimes\densityT{t}$ on $[0,T]\times\dom$ (with $\densityT{t}(\dom)=\text{const.}$ for a.e.\ $t\in[0,T]$ and $\dom\subset\domDelta$) by more regular ones, $\densityN{n}=\d t\otimes\densityTN{t}{n}$, such that the curve $t\mapsto\densityTN{t}{n}$ is Hölder-$\frac{1}{2}$ continuous in the Wasserstein-2 space. Based on \cite[Thm.\,5.14]{OTAppliedMath} we can also ensure that $\BBEnergy(\densityN{n},\momentumN{n})\le\frac{1}{\delta_n}$ for any sequence $\delta_n\to 0$, where the sequence of $\R^3$-valued measures $\momentumN{n}$ is constructed such that $(\densityN{n}, \momentumN{n})$ satisfies the continuity equation \eqref{eqn:CE_introduction}. Additionally, we give the approximation result $\mathbb W_2(\densityTN{t_n}{n},\densityT{t})\to0$ for the Wasserstein-2 distance $\mathbb W_2$ along a subsequence for almost every $t$ and $t_n\to t$.
This approximation result will be needed in our $\Gamma$-convergence analysis, but is more generally applicable whenever analysing optimal transport-based regularization.

\item In \cref{sec:ForwardOperator} we introduce the PET forward model based on \cite{Schmitzer_DynamicCellImaging,PET_Base} and slightly adapt and generalize it to our setting.

\item \Cref{SectionGammaConvergence} shows stochastic $\Gamma$-convergence of the discrete PET reconstruction model to a (continuous) limit model for an increasing signal intensity $q_n\to\infty$ while allowing for resolvable time differences, detector sizes, and regularization parameter to go to zero. The $\Gamma$-limit (there remains no stochasticity of the PET measurements, and the reconstruction is deterministic) is basically a continuous Kullback--Leibler divergence, i.e.\ our convergence result motivates using a continuous Kullback--Leibler divergence as data term for high resolution data.

Additionally, a classical convergence of minimizers result is shown. Specifically, we show that if all sources of noise (discretization and measurement noise) vanish in the limit, then any sequence of minimizers converges to the ground truth which is the measure that generates the PET measurements.

For the convergence result to hold, either the regularization parameter $\beta$ or the bin size of time-binned measurements must decrease more slowly than the radioactivity $q$ increases.
Such relation is expected for inverse problems, the interpretation here is as follows:
The measured coincidences all happen at different time points and only become related to each other via time binning or via the temporal Benamou--Brenier regularization.
If this relation becomes too weak (e.g.\ due to too low regularization weight) the radioactive material may have moved arbitrarily in between the coincidences
so that it can no longer be localized (since localization requires multiple coincidences).
\end{itemize}

\subsection{Preliminaries and notation}
We start with introducing some notation, part of which we actually already used above.
The Banach space of Radon measures on a compact domain $X$ will be denoted $\meas(X)$ with norm $\norm{\cdot}$, the subset of nonnegative measures by $\measp(X)$. On $\meas(X)$ we have the weak-* convergence $\mu_n\convStar\mu$.
For two measures $\mu,\alpha\in\meas(X)$ with $\mu$ absolutely continuous with respect to $\alpha$, the Radon--Nikodym derivative of $\mu$ with respect to $\alpha$ is denoted $\RNderivative{\mu}{\alpha}$.
The restriction of a measure $\mu$ to some $\mu$-measurable set $S$ is denoted $\mu\restr S$,
and the pushforward of $\mu$ under some $\mu$-measurable map $f$ is denoted $\pushforward{f}\mu$.
By $\mathcal{L}^d$ and $\hd^d$ we denote the $d$-dimensional Lebesgue and Hausdorff measure, where for $d=1$ we may drop the exponent,
and $\delta_a$ denotes the Dirac measure at some point $a$.
Sometimes we will for simplicity also refer to the Lebesgue measure in time by $\d t$.
Furthermore, we will indicate random variables by boldface letters such as $\bm E$
while their realizations have normal font, thus $E=\bm E(\omega)$ for $\omega$ a random element of the standard probability space $(\Omega,\mathcal{F},P)$.
Finally, given a measure $\lambda$, by $\Poi{\lambda}$ we denote the Poisson point process with intensity $\lambda$.
We will only consider $\sigma$-finite and diffuse intensities on (metric) Borel spaces so that the corresponding Poisson point processes are proper and simple and thus can be interpreted as random sets of points
(see \cite{bookPPP_last_penrose, bookPPP_kingman} for an introduction to Poisson point processes).

$\mathrm{L}^p$, $p\ge1$, denotes the standard Lebesgue $\mathrm{L}^p$-space, $f_n\xrightharpoonup{\mathrm{L}^p} f$ denotes weak convergence in $\mathrm{L}^p$, and $C$ and $C^1$ ($C_c^1$) denote continuous and continuously differentiable (and compactly supported) functions.

We will further employ the notation $a\lesssim b$ to indicate the existence of an independent constant $c>0$ such that $a\leq cb$
(analogously, $b\gtrsim a$ stands for $a\lesssim b$ and $a\simeq b$ for $a\lesssim b$ and $b\lesssim a$). We use $C$ for a constant that may change its value in consecutive estimates. Moreover, we use the little-o notation $f_n\in o(g_n)$ meaning that $\lim_{n}\frac{f_n}{g_n}=0$. 

The spatial setting of the PET reconstruction model is as follows:
The sought radionuclide distribution is confined to $\dom\subset\R^3$, the closure of a bounded, open and convex set with $0\in\mathrm{int}(\dom)$.
The PET scanning tube $\domDelta\supset\dom$ is compact and convex such that $\text{dist}(\dom, \partial\domDelta)\ge\delta>0$ (the detectors are located in $\partial\domDelta$).
For the reader's convenience below we provide a reference list of further model-specific symbols and quantities frequently used throughout the article.

\setlength\extrarowheight{3pt}
\begin{longtable}{@{}lp{12.5cm}}
	$A^\att, \forwardOp^\sct, \forwardOp^\dt$ & Forward operators describing attenuation, scattering and normal detection. They are either defined on time slices, i.e.\ on $\meas(\dom)$, or on $\meas(\spacetime)$ via $A^{a/c/d}\density = \d t\otimes A^{a/c/d}\densityT{t}$.\\
	
	$A$, $\forwardP$ & Total (weighted) forward operator $A = p^\sct \forwardOp^\sct \!+\! p^\dt \forwardOp^\dt$, $A^u = \UF p^\sct \forwardOp^\sct \!+\! p^\dt \forwardOp^\dt$.  \\
	
	$\forwardBinnedOp^\UF$, $\forwardBinnedOp^\UF_n$ & (Weighted) discrete forward operator where the subscript $n$ denotes a dependence on system quantities such as detector size, see \eqref{eq:DefDiscreteForwardOperator}.\\

	$\dom\subset\R^3$ & Compact and convex set where the radioactive material stays. \\
	
	$\domDelta\subset\R^3, \delta$ & Compact and convex set such that $\dom\subset\domDelta$ and $\mathrm{dist}(\dom,\partial\domDelta)\ge\delta$ for some $\delta>0$. The detectors are located at the boundary $\partial\domDelta$.\\

	$\domDeltaHalf\subset\R^3$ & It is $\domDeltaHalf=D+B_{\delta/2}(0)$. Tracer densities are supported on $\domDeltaHalf$ after smoothing with positron range kernel.\\
	$\bm E_q$, $E_q$, $\norm{E_q}$ & Measurement $E_q$, realization of a Poisson point process $\bm E_q$ with intensity measure $\frac{1}{q}\forwardOp\density^\dagger$. To be interpreted as either a set or equivalently as a discrete empirical measure. $\norm{E_q}$ denotes the number of elements in the set.\\
	
	$\mathbb{E}[\bm N]$ & Expectation of the random variable $\bm N$. \\

	$G$ & Smooth, compactly supported convolution kernel $G\colon B_{\delta/2}(0)\to[0,\infty)$ describing the probability density of the annihilation location of a positron emitted at the origin. \\
	
	$\Gamma^k\subset\partial\domDelta$, $M$ & Discrete detectors and number of detectors. For $k\neq l$ we have the detector pairs $\Gamma^k\times\Gamma^l$ which register photon pairs.\\
		
	$\hd^d$ & $d$-dimensional Hausdorff measure. \\
	
	$\momentum, \momentum^\dagger\!\!\in\!\meas(\timeInterval\!\!\times\!\!\dom)^3$         & Measures describing the material flux corresponding to the temporal variation of the mass distribution $\density,\density^\dagger$.     \\
		
	$\meas(X)$, $\meas(X)^3$ & Space of ($\R^3$-valued) Radon measures on $X$.\\

	$\measp(X)$, $\meas_c(X)$ & Nonnegative Radon measures and those with constant mass in time (see \cref{thm:closednessOfBSet}).\\

	$(\probSpace,\sigmaProb,\mathbb{P})$ & Standard probability space on which the random variables are defined. \\

	$\absNorm{\cdot}$ & Euclidean norm.\\
	
	$\mathds{1}_C$ & Characteristic function of the $C$, i.e.\ $\mathds{1}_C(x)=1$ for $x\in C$ and $0$ else. \\
	
	$P$ & X-ray transform, see \cref{sec:ForwardOperator}.\\
	
	$\Poi{\mu}$ & Poisson point process with intensity measure $\mu$. \\

	$p^\att, p^\sct, p^\dt$ & Probabilities for attenuation, scattering and normal detection. It holds $p^\att+p^\sct+p^\dt=1$.\\
		
	$\nu$& $\nu=\dt t\otimes(\hd^2\restr\partial\domDelta)\otimes(\hd^2\restr\partial\domDelta)$. The forward operator $\dt t\otimes \forwardOp\densityT{t}$ is absolutely continuous w.r.t.\ $\nu$.\\
	
	$\UF$ & Lagrange parameter $\UF>0$ that weighs the influence of the scatter part of the forward operator. It holds $\forwardOpP = \UF p^\sct \forwardOp^\sct + p^\dt \forwardOp^\dt$.  \\
	
	$R$ & $R\colon\domDeltaHalf\times S^2\to\boundDomDeltaSq,\;
	R(x,v)=\partial\domDelta\cap(x+\R v)\,,$ is the measurement function that maps a point $x$ (where an annihilation happened) and a direction $v$ onto the photon pair's detection location. It is comparable to the classical Radon transform, see \eqref{eq:RFunctionForwardOperator}.\\

	$\density, \density^\dagger\!\!\in\!\measp(\timeInterval\!\!\times\!\!\dom)$ &        Measures describing radionuclide distribution in spacetime. $\density^\dagger=\d t\otimes\density^\dagger_t\in\meas_c(\spacetime)$ represents the ground truth tracer distribution.    \\
	
	$S^2$ & Sphere in $\R^3$, i.e.\ $S^2 = \{x\in\R^3 \ | \ \absNorm{x}=1\}$. \\
	
	
	$[0,T]$ & Time interval during which the measurements are taking place.\\
	
	$\tau^i\subset[0,T]$, $N$ & Discrete time intervals, $i=1,\ldots,N$.\\
	
	



	
	
	
	
	$\mathbb{V}[\bm N]$ & Variance of the random variable $\bm N$.\\
	
	$\mathbb{W}_p(\mu, \alpha)$ & Wasserstein-$p$ distance between the nonnegative measures $\mu$ and $\alpha$ with equal mass.\\


	


	
	

\end{longtable}

\section{Poisson Point Processes}\label{SectionPPP}
In this section we provide convergence results for Poisson point processes (PPPs) with intensities tending to infinity. More precisely: We investigate the limit behaviour of $\frac{1}{q_n}\bm E_{q_n}$ for a PPP $\bm E_{q_n}=\Poi{q_n\measPPPSection}$ and $q_n\to\infty$.
Radioactive decay can well be described by PPPs: The points of the realization of a (suitably modelled) PPP can be seen as the locations in spacetime of radioactive decays. We start by defining PPPs (see \cite{bookPPP_last_penrose, bookPPP_kingman} for more details).

\begin{definition}[Poisson point process, \cite{bookPPP_last_penrose}]
Let $(X, \mathcal{X})$ be a measurable space and $\measPPPSection$ an (s-)finite measure on $X$. A \emph{Poisson point process} with intensity measure $\measPPPSection$ is a point process $\bm N$ on $X$ satisfying the following:
\begin{itemize}
\item For $B\in\mathcal{X}$ the distribution of $\bm N(B)$ is Poisson with parameter $\measPPPSection(B)$.
\item For every $m\in\N$ and any pairwise disjoint sets $B_1,\ldots,B_m\in\mathcal{X}$ the random variables $\bm N(B_1),\ldots,\bm N(B_m)$ are independent.
\end{itemize}
\end{definition}


We now investigate the limit process. We assume that the intensity measure $\measPPPSection$  is finite.
For the processes $\bm E_{q_n}$ we consider two different settings that correspond to two different ways of defining the sequence of PPPs.

In the first setting we choose an arbitrary (e.g.\ independent) sequence of PPPs $\bm E_{q_n}$ with intensity measure $q_n\measPPPSection$. This corresponds to the (theoretical or thought) experiment in which for each $n$ the PET measurement is repeated (with the only difference of a higher intensity, e.g.\ realized by a shorter radionuclide halflife), discarding all previous measurements.
In case of an independent sequence of PPPs the points drawn in step $n$ according to the law of $\bm E_{q_n}$ are therefore independent of the points drawn in step $n-1$.
The increasing SNR is here achieved by higher radioactivity.

In the second setting we consider instead a strongly coupled sequence of processes: Points from the previous step are not discarded, but new points are added in such a way that $\bm E_{q_n}$ is still a PPP with intensity measure $q_n\measPPPSection$.
This corresponds to an experiment in which for each $n$ the PET measurement is repeated (with radionuclides of different or the same halflife) and all measurements so far are combined.
The increasing SNR is here achieved by combining the repeated measurements.
The random variables corresponding to this situation are defined as follows using a so called stochastic coupling (see \cite[Chp.\,3, Sec.\,2 and 3]{stochCoupling}, \cite[Sec.\,3.1]{stochCoupling2}): For each realization of the random variables we start with an auxiliary infinite point configuration and define our random variables roughly as a truncation of the infinite point list. With less truncation the intensity of the PPPs grows. In detail, for a measurable space $(X,\mathcal X)$ we introduce
\begin{align*}
\mathbb{X}=X\times[0,\infty)\quad\text{and}\quad\mathbb{X}_q=X\times[0,q].
\end{align*}
We now define $\bm Y$ to be a PPP on $\mathbb{X}$ with intensity measure $\gamma=\measPPPSection\otimes\mathcal{L}$ and let $\bm Y_q=\bm Y\restrictTo{\mathbb{X}_q}$. Then $\bm Y_q$ is a PPP on $\mathbb{X}_q$ with intensity measure $\gamma_q=(\measPPPSection\otimes\mathcal{L})\restr \mathbb{X}_q$ (note that $\gamma_q$ is finite). Finally, let $\pi^q_X\colon\mathbb{X}_q\to X$ be the projection onto $X$. Then we define our PPPs $\bm E_{q_n}$ via $\bm E_{q_n}(\omega)(C)\coloneqq\bm Y_{q_n}(\omega)((\pi_X^{q_n})^{-1}(C))$, $\omega\in\Omega$, $C\in\mathcal{X}$. By the mapping theorem \cite[Thm. 5.1]{bookPPP_last_penrose} $\bm E_{q_n}$ is a PPP on $(X, \mathcal{X})$ with intensity measure $q_n\measPPPSection=\pushforward{\pi_X^{q_n}}{\gamma_{q_n}}$ as desired.

For the convergence result we make use of the following theorem.
\begin{theorem}[Concentration inequality {\cite[Cor.\,2]{PPPconcentrationInequality}}]\label{thm:PPPEstim}
Let $\bm N$ be a Poisson point process on some measurable space $(X,\mathcal{X})$ with finite intensity measure $\measPPPSection$ without atoms, and let $\{f_i\}_{i\in \mathcal I}$ be a countable family of functions with values in $[-b,b]$. We define
\begin{align*}
	\bm Z\coloneqq\sup_{i\in \mathcal I}\absNorm{\int f_i(\d \bm N-\d \measPPPSection)} \quad\text{and}\quad\measPPPSection_0\coloneqq\sup_{i\in \mathcal I}\int f_i^2\dint \measPPPSection.
\end{align*}
Then for all $\varepsilon,x>0$ and with $\kappa(\varepsilon)=5/4+32/\varepsilon$ it holds
\begin{align*}
\Prob{\bm Z\ge(1+\varepsilon)\Expect{\bm Z}+\sqrt{12\measPPPSection_0x}+\kappa(\varepsilon)bx}\le\exp(-x).
\end{align*}
\end{theorem}

In our main convergence result, the following \cref{thm:stochConv}, we used ideas from \cite{PoissonData}.

\begin{theorem}[Convergence of PPP]\label{thm:stochConv}
	Consider a measurable space $(X,\mathcal{X})$ with a finite measure $\measPPPSection$ without atoms, and a sequence of finite disjoint partitions $(C_n^k)_{k,n}$, i.e.\ $\bigcup_{k=1}^{K_n}C_n^k=X$ (the number of sets $K_n$ grows monotonously, but may stay bounded), as well as monotone sequences $q_n\to\infty$ and $r_n\to r\in[0,\infty)$ with $\sqrt{K_n/q_n}\in o(r_n)$. Moreover, let $\bm E_{q_n}$ be a PPP on $(X,\mathcal{X})$ with finite intensity measure $q_n\measPPPSection$. We distinguish two settings:
	\begin{enumerate}[label=(\alph*)]
		\item $\bm E_{q_n}$ is an arbitrary sequence of PPPs with intensity $q_n\measPPPSection$, and $\sqrt{\log(n)/q_n}\in o(r_n)$. \label{item:PPP1}
		\item $\bm E_{q_n}$ is defined via stochastic coupling and $\sqrt{\log\log(q_n)/q_n}\in o(r_n)$.
		Moreover we assume that the partitions are nested, i.e.\ for every $k\in\left\{1,\ldots,K_n\right\}$ there exists $k'\in\left\{ 1,\ldots K_{n-1}\right\}$ with $C_n^k\subset C_{n-1}^{k'}$. \label{item:PPP2}
	\end{enumerate}
	If \ref{item:PPP1} or \ref{item:PPP2} holds, then $\sum_{k=1}^{K_n}\absNorm{\frac{1}{q_n}\bm E_{q_n}(C_n^k)-\measPPPSection(C_n^k)}$ almost surely tends to zero with rate $r_n$,
	\begin{align*}
	\frac{1}{r_n}\sum_{k=1}^{K_n}\absNorm{\frac{1}{q_n}\bm E_{q_n}(C_n^k)-\measPPPSection(C_n^k)}\xrightarrow[n\rightarrow\infty]{\text{a.s.}}0.
	\end{align*}
\end{theorem}

\begin{remark}[Improvement with coupling]\label{rmk:StochConvCoupling}
	Case \ref{item:PPP2} only improves on \ref{item:PPP1} if $q_{n}$ has subexponential growth, i.e.\ $\log q_n\in o(n)$, because otherwise $\sqrt{\log\log(q_n)/q_n}\gtrsim\sqrt{\log(n)/q_n}$.
\end{remark}
	
\begin{proof}
	We want to show that
	\begin{align*}
	\bm Z_n\coloneqq\frac{1}{r_nq_n}\sum_{k=1}^{K_n}\absNorm{\bm E_{q_n}(C_n^{k})-q_n\measPPPSection(C_n^{k})}
	\end{align*}
	converges to zero almost surely. We start with situation \ref{item:PPP1} and define the set of functions $A_n=\left\{ \sum_{k=1}^{K_n}\alpha_{k}\mathds{1}_{C_n^{k}} \ | \ \alpha_{k}\in\{\pm 1\} \right\}$. These functions can be used to rewrite the random variable $\bm Z_n$ in order to be able to apply \cref{thm:PPPEstim}. We find
	\begin{equation*}
	r_nq_n\bm Z_n
	\!=\!\sum_{k=1}^{K_n}\!\max_{\alpha\in\{ \pm 1 \}}\!\alpha\!\left( \bm E_{q_n}(C_n^{k})\!-\!q_n\measPPPSection(C_n^{k}) \right)
	\!=\!\max_{\phi\in A_n}\int\!\!\phi\!\left( \d \bm E_{q_n}\!-\!q_n\d \measPPPSection \right)
  \!=\! \max_{\phi\in A_n}\absNorm{\int\!\!\phi\!\left( \d \bm E_{q_n}\!-\!q_n\d \measPPPSection \right)},
	\end{equation*}
	where the last equation is true by the definition of $A_n$. Next, we want to bound $\Expect{r_nq_n\bm Z_n}$ appropriately. First, $\Expect{r_nq_n\bm Z_n}\le\sqrt{\Expect{r_n^2q_n^2\bm Z_n^2}}$ by Jensen's inequality. With Hölder's inequality for sums we get
	\begin{multline*}
	\Expect{r_n^2q_n^2\bm Z_n^2}
	\le\Expect{\left(\textstyle \sum_{k=1}^{K_n}|\bm E_{q_n}(C_n^{k})-q_n\measPPPSection(C_n^{k})| \right)^2}
	\le K_n\sum_{k=1}^{K_n}\Expect{\left( \bm E_{q_n}(C_n^{k})-q_n\measPPPSection(C_n^{k}) \right)^2}\\
	=K_n\sum_{k=1}^{K_n}\mathbb{V}[\bm E_{q_n}(C_n^{k})]
	=K_n\sum_{k=1}^{K_n}q_n\measPPPSection(C_n^{k})
	= K_nq_n\normS{\measPPPSection}.
	\end{multline*}
	Thus, we have $\Expect{r_nq_n \bm Z_n}\le\sqrt{ K_nq_n\norm{\measPPPSection}}$. Next, \cref{thm:PPPEstim} (applied to $\varepsilon=1$, $\measPPPSection_0= q_n\normS{\measPPPSection}$, and $b=1$) yields for $x\ge 1$
	\begin{align*}
	\exp\left( -x \right)&\ge\Prob{r_nq_n\bm Z_n\ge 2\Expect{r_nq_n\bm Z_n}+\sqrt{12x q_n \norm{\measPPPSection}}+\kappa(1)x}\\
	&\ge\Prob{r_nq_n\bm Z_n\ge C(\sqrt{q_nK_n} + \sqrt{q_nx}+x)}
	\end{align*}
	or equivalently
	\begin{equation*}\textstyle
	\Prob{\bm Z_n\ge C\left( \sqrt{\frac{K_n}{r_n^2q_n}}+\sqrt{\frac{x}{q_nr_n^2}}+\frac{x}{q_nr_n}\right)}\le\exp(-x)
	\end{equation*}
	for some $C>0$. Next, let $1>\Delta>0$. Due to $\sqrt{K_n/q_n}\in o(r_n)$ we can derive an expression for $\Prob{\bm Z_n\ge\Delta}$. Consider $n$ large such that $C\sqrt{K_n/(r_n^2q_n)}\le\Delta/2$, and pick $x\simeq\Delta^2q_nr_n^2$ such that $\sqrt{x/(q_nr_n^2)}+x/(q_nr_n)\leq\Delta/(2C)$. Inserting this into the above inequality we arrive at
	\begin{align*}
	\Prob{\bm Z_n\ge\Delta}\le\exp\left(-C\Delta^2q_nr_n^2\right).
	\end{align*}
	If this expression is summable for every $\Delta$, then $\bm Z_n$ converges almost surely to zero \cite[Thm.\,6.12]{KlenkeProbabilityTheory}. Summability is guaranteed if for $n$ large and for some $\delta>0$ it holds
	\begin{equation*}
	\exp\left(-C\Delta^2q_nr_n^2\right)\le\left(\tfrac{1}{n}\right)^{1+\delta}
	\qquad\Longleftrightarrow\qquad
	\frac{C\Delta^2}{1+\delta}\ge\frac{\log(n)}{q_nr_n^2},
	\end{equation*}
	i.e.\ if $\sqrt{\log(n)/q_n}\in o(r_n)$.
	
	Next, we prove the statement for \ref{item:PPP2}.
	We only consider $r_n\to 0$ since in the other cases $r_n$ can up to a bounded factor be treated like a constant, simplifying all estimates.

	We proceed in a similar way as before, but make use of the special modelling of the random variables via stochastic coupling. For $m\le n$ we define the sets
	\begin{align*}
	\tilde{C}_{n,m}^{k}=\left\{ (x,r)\in\mathbb{X} \ | \ x\in C_n^{k},\ r\le q_m \right\},
	\end{align*}
	i.e.\ the parameter $m$ controls the size of the sets $\tilde{C}_{n,m}^{k}$, which will correspond to controlling the intensity of the PPPs. Next, we can write
	\begin{align*}
	r_nq_n\bm Z_n=\sum_{k=1}^{K_n}\absNorm{\bm E_{q_n}(C_n^{k})-q_n\measPPPSection(C_n^{k})}=\sum_{k=1}^{K_n}\big|\bm Y_{q_n}(\tilde{C}_{n,n}^{k})-\int_{\tilde{C}_{n,n}^{k}}1\dint \gamma_{q_n}\big|.
	\end{align*}
	For $y>1$ we consider the set of indices $Q_n=\left\{ m\in\mathbb{N} \ | \ y^{n-1}<q_m\le y^n \right\}$ and their maximum ${l(n)}=\max Q_n$ for every $n$ with $Q_n\neq\emptyset$. For every $m\in Q_n$ we find (using a similar approach as in the proof of \cite[Thm.\,5.29]{KlenkeProbabilityTheory}), due to the monotonicity of $r_n$ and the assumption $C_n^{k}\subset C_{n-1}^{k'}$ (and hence $\tilde{C}_{n,m}^{k}\subset \tilde{C}_{n-1, m}^{k'}$),
	\begin{align*}
	\bm Z_{m}
	&=\frac{1}{r_{m}q_{m}} \sum_{k=1}^{K_{m}}\absNorm{\bm Y_{q_{m}}(\tilde{C}_{{m},{m}}^{k})-\int_{\tilde{C}_{{m},{m}}^{k}}1\dint \gamma_{q_{m}}}\le\frac{1}{r_{l(n)}y^{n-1}}\max_{l\in Q_n}\sum_{k=1}^{K_l}\absNorm{\bm Y_{q_{l(n)}}(\tilde{C}_{l,l}^{k})-\int_{\tilde{C}_{l,l}^{k}}1\dint \gamma_{q_{l(n)}}} \\
	&\le\frac{1}{r_{l(n)}y^{n-1}}\sum_{k=1}^{K_{l(n)}}\max_{1\le l\le{l(n)}}\absNorm{\bm Y_{q_{l(n)}}(\tilde{C}_{{l(n)},l}^{k})-\int_{\tilde{C}_{{l(n)},l}^{k}}1\dint \gamma_{q_{l(n)}}}\\
	&=\frac{1}{r_{l(n)}y^{n-1}}\sum_{k=1}^{K_{l(n)}}\max_{1\le l\le{l(n)}}\max_{\alpha\in\{\pm 1\}}\alpha \left(\bm Y_{q_{l(n)}}(\tilde{C}_{{l(n)},l}^{k})-\int_{\tilde{C}_{{l(n)},l}^{k}}1\dint \gamma_{q_{l(n)}}\right)\\
	&=\frac{1}{r_{l(n)}y^{n-1}}\max_{\phi\in A_{l(n)}}\int\phi\left( \d  \bm Y_{q_{l(n)}} - \d \gamma_{q_{l(n)}}\right)
	=\frac{1}{r_{l(n)}y^{n-1}}\max_{\phi\in A_{l(n)}}\absNorm{\int\phi\left( \d  \bm Y_{q_{l(n)}} - \d \gamma_{q_{l(n)}}\right)}\eqqcolon \tilde{\bm Z}_n
	\end{align*}
	for the set of functions
	\begin{align*}
	A_{n}=\left\{ \sum_{k=1}^{K_n}\alpha_{k}\mathds{1}_{\tilde{C}_{n,m}^{k}} \ | \ \alpha_{k}\in\{ \pm 1\},\ m\le n \right\}.
	\end{align*}
	We now show almost sure convergence of $\tilde{\bm Z}_n$ to zero which gives us the desired convergence of $\bm Z_{n}$ to zero almost surely. Similar to situation \ref{item:PPP1}, we estimate
	\begin{align*}
	\Expect{\left(r_{l(n)}y^{n-1}\tilde{\bm Z}_n\right)^2}
	&= \Expect{\left(\max_{\phi\in A_{l(n)}} \int\phi\left(\d \bm Y_{q_{l(n)}}-\d \gamma_{q_{l(n)}}\right) \right)^2}\\
	&\le K_{l(n)}\sum_{k=1}^{K_{l(n)}}\Expect{\max_{1\le l\le{l(n)}}\left( \bm Y_{q_{l(n)}}(\tilde{C}_{{l(n)},l}^{k}) -\int_{\tilde{C}_{{l(n)},l}^{k}}1\dint \gamma_{q_{l(n)}}\right)^2}.
	\end{align*}
	Notice that due to the definition of $\bm Y$ we can write $\bm Y_{q_{l(n)}}(\tilde{C}_{{l(n)},l}^{k})$ as a sum
	\begin{equation*}
	\bm Y_{q_{l(n)}}(\tilde{C}_{{l(n)},l}^{k})
	=\sum_{j=1}^l\left(\bm Y(\tilde{C}_{{l(n)},j}^{k}\setminus\tilde{C}_{{l(n)},j-1}^{k}) \right),
	\end{equation*}
	whose summands are independent random variables since the sets $(\tilde{C}_{{l(n)},j}^{k}\setminus\tilde{C}_{{l(n)},j-1}^{k})_j$ are disjoint and $\bm Y$ is a PPP.
	This makes
	\begin{equation*}
	\bm B^k_l
	\coloneqq\bm Y_{q_{l(n)}}(\tilde{C}^k_{l(n),l})-\int_{\tilde{C}^k_{l(n),l}}1\dint\gamma_{l(n)}
	=\sum_{j=1}^l\left(\bm Y\left(\tilde{C}_{{l(n)},j}^{k}\setminus\tilde{C}_{{l(n)},j-1}^{k}\right)-\int\mathds{1}_{\tilde{C}_{{l(n)},j}^{k}\setminus\tilde{C}_{{l(n)},j-1}^{k}}\dint\gamma\right),
	\end{equation*}
	$l=1,\ldots,l(n)$, a martingale \cite[Exm.\,9.30]{KlenkeProbabilityTheory} so that Doob's $L^p$ inequality \cite[Thm.\,11.2]{KlenkeProbabilityTheory} is applicable. Applying the inequality for $p=2$ yields
	\begin{equation*}
	\Expect{\left( \max_{ 1\le l\le l(n)}\absNorm{\bm B^k_l} \right)^2}\le\left(\tfrac{2}{2-1}\right)^2\Expect{\absNorm{\bm B^k_{l(n)}}^2}=4\mathbb{V}\left[\bm Y(\tilde{C}_{{l(n)},{l(n)}}^{k}) \right]=4q_{l(n)}\measPPPSection(C_{l(n)}^{k}).
	\end{equation*}
	This leads to the estimate
	$
	\mathbb{E}[(r_{l(n)}y^{n-1}\tilde{\bm Z}_n)^2]
	\le 4 K_{l(n)}q_{l(n)}\norm{\measPPPSection}
	$
	and thus by Jensen's inequality to
	\begin{equation*}
	\Expect{r_{l(n)}y^{n-1}\tilde{\bm Z}_n}\le 2\sqrt{K_{l(n)}q_{l(n)}\norm{\measPPPSection}}.
	\end{equation*}
	\Cref{thm:PPPEstim} (applied to $\varepsilon=1$, $\measPPPSection_0=q_{l(n)}\norm{\measPPPSection}$ and $b=1$) yields for $x\ge 1$ and some $C>0$
	\begin{multline*}
	\exp(-x)
	\ge\Prob{r_{l(n)}y^{n-1}\tilde{\bm Z}_n\ge 2\Expect{r_{l(n)}y^{n-1}\tilde{\bm Z}_n}+\sqrt{12x q_{l(n)}\norm{\measPPPSection}}+x\kappa(1)}\\
	\ge\Prob{r_{l(n)}y^{n-1}\tilde{\bm Z}_n\ge C\left(\sqrt{K_{l(n)}y^{n}}+\sqrt{y^nx}+x\right)}
	=\Prob{\tilde{\bm Z}_n\ge \tfrac{Cy}{r_{l(n)}}\left(\sqrt{\tfrac{K_{l(n)}}{y^{n}}}+\sqrt{\tfrac{x}{y^n}}+\tfrac{x}{y^n}\right)},
	\end{multline*}
	where we used $q_{l(n)}\le y^n$. Now let $1>\Delta>0$ and note that the condition $\sqrt{K_n/q_n}\in o(r_n)$ implies $\sqrt{K_{l(n)}/y^n}/r_{l(n)}\to0$, which allows to establish a bound on $\Prob{\tilde{\bm Z}_n\ge\Delta}$.
	Taking $n$ large enough such that $Cy\sqrt{K_{l(n)}/y^n}/r_{l(n)}\le\Delta/2$ and picking $x\simeq\Delta^2y^nr_{l(n)}^2$ such that $Cy(\sqrt{x/y^n}+x/y^n)/r_{l(n)}\leq\Delta/2$, we get
	\begin{equation*}\label{eqn:StochConv}
	\Prob{\tilde{\bm Z}_n\ge\Delta}
	\le\exp\left( -\Delta^2Cy^nr_{l(n)}^2 \right)
	=\left(\!\frac{1}{n\log y} \!\right)^{\frac{\Delta^2Cy^nr_{l(n)}^2}{\log\log(y^n)}}
	\end{equation*}
	for some constant $C>0$. If for every $\Delta$ this expression is summable (over all $n$ for which $Q_n\neq\emptyset$ and thus $\tilde{\bm Z}_n$ is well-defined), then $\bm Z_m\le\tilde{\bm Z}_n$ converges to zero almost surely.
	A sufficient condition is that the exponent tends to infinity as $n\to\infty$
	(if all $Q_n$ are nonempty, this would also be necessary).
	Due to $y^{n-1}<q_{{l(n)}}\le y^n$, the exponent can, up to a constant factor, be bounded above and below by $\Delta^2q_{l(n)}r_{l(n)}^2/\log\log(q_{l(n)})$,
	so the condition $\sqrt{\log\log(q_n)/q_n}\in o(r_n)$ indeed suffices for the desired summability.
\end{proof}

A direct consequence is a convergence rate in the flat norm of PPPs with increasing intensity as stated in the remainder of this section.
Recall that the Assouad dimension $\mathrm{dim}_A(X)$ of a bounded metric space $X$ can be defined as the infimal number $a\in[0,\infty]$
such that a (metric) ball of diameter $r$ can be covered by no more than $C(r/s)^a$ many balls of diameter $s<r$ with $C\geq1$ a constant independent of $r$, $s$ \cite{DefinitionAssouadDimension}
(for many spaces, the Assouad dimension simply coincides with the Hausdorff or the Minkowski dimension).

\begin{corollary}[Convergence rate in flat distance]\label{thm:flatConvergence}
Let $X$ be a bounded locally compact metric space with Assouad dimension $\mathrm{dim}_A(X)< a<\infty$,
and let $\bm E_{q_n}$ be a PPP on $X$ with finite intensity measure $q_n\measPPPSection\in\measp(X)$ without atoms.
Then almost surely $\frac{1}{q_n}\bm E_{q_n}$ converges weakly-* to $\measPPPSection$.
Moreover, the convergence rate in flat distance is any rate $s_n$ satisfying $\max\{q_n^{-1/(a+2)},\sqrt{\log n/q_n}\}\in o(s_n)$.
If $\bm E_{q_n}$ is defined via stochastic coupling, the convergence rate is even $s_n=q_n^{-1/(a+2)}$.
\end{corollary}
\begin{proof}
Since $\measPPPSection$ is finite, $\bm E_{q_n}$ is almost surely finite and we consider only those realizations of the random variables. For $\omega\in\Omega$ let $E_{q_n}= \bm E_{q_n}(\omega)$ be such a realization. The flat distance between $\measPPPSection$ and $\frac{1}{q_n}E_{q_n}$ is computed by
\begin{equation*}
\mathrm{d}_{\mathrm{flat}}(\measPPPSection,\tfrac{1}{q_n}E_{q_n})=\inf_{\substack{\mu\in\measp(X) \\ \mu(X) = \measPPPSection(X)}}\left(\mathbb{W}_1(\measPPPSection,\mu)+\normS{\mu-\tfrac{1}{q_n}E_{q_n}}\right)
\end{equation*}
with $\mathbb{W}_1$ the Wasserstein-1 distance.
Now consider the measures
\begin{equation*}
	\mu_n=\sum_{k=1}^{K_n}\frac{\measPPPSection(C_n^k)}{E_{q_n}(C_n^k)}E_{q_n}\restr C^k_n
\end{equation*}
for some sequence of partitions $(C_n^k)_{k,n}$ of $X$ into $K_n$ disjoint subsets of diameter no larger than $s_n$.
We have $\mu_n(X)=\measPPPSection(X)$ as well as $\normS{\mu_n-\frac{1}{q_n}E_{q_n}}=\sum_{k=1}^{K_n}\absNormS{\measPPPSection(C_n^k)-\frac{1}{q_n}E_{q_n}(C_n^k)}$
so that \cref{thm:stochConv} applies to this expression.
To estimate $\mathbb{W}_1(\measPPPSection,\mu_n)$ we construct an admissible transport plan: Let $\transportPlan_k^n\in\measp(C_n^k\times C_n^k)$ be the optimal transport plan for the transport of $\measPPPSection\restr C_n^k$ to $\mu_n\restr C_n^k$ and define $\transportPlan^n=\sum_{k=1}^{K_n}\transportPlan_k^n\in\measp(X\times X)$ (we extend each $\transportPlan_k^n$ onto $X\times X$ by zero).
Then the marginals of $\transportPlan^n$ are $\measPPPSection$ and $\mu_n$ so that
\begin{equation*}
\mathbb{W}_1(\measPPPSection,\mu_n)\le\int_{X\times X}\dist(x,y)\dint\transportPlan^n(x,y)
\le\sum_{k=1}^{K_n}\text{diam}(C_n^k)\measPPPSection(C_n^k)
\le s_n\norm{\measPPPSection}.
\end{equation*}
Invoking \cref{thm:stochConv} with rate $r_n$, we get for almost every $\omega\in\Omega$, $E_{q_n}=\bm E_{q_n}(\omega)$,
\begin{align*}
	\mathrm{d}_{\mathrm{flat}}(\measPPPSection,\tfrac{1}{q_n}E_{q_n})\lesssim s_n+\sum_{k=1}^{K_n}\absNormS{\measPPPSection(C_n^k)-\tfrac{1}{q_n}E_{q_n}(C_n^k)}\lesssim s_n+r_n.
\end{align*}
For an optimal rate, we need to minimize this expression under the constraints imposed by \cref{thm:stochConv}.
Since $X$ has Assouad dimension $\mathrm{dim}_A(X)< a$,
we can choose the sequence of partitions $(C_n^k)_{k,n}$ such that $K_n\in o(s_n^{-a})$.
The minimization now leads to the choice $r_n=s_n$ with $s_n$ satisfying the growth condition of the statement.
It is straightforward to check that this choice satisfies the conditions $\sqrt{K_n/q_n}\in o(r_n)$ and $\sqrt{\log n/q_n}\in o(r_n)$ from \cref{thm:stochConv}.

To cover the case of stochastic coupling we have to use nested partitions. Indeed (as shown further below) 
one can even choose the partitions $(C_n^k)_{k,n}$ to satisfy $K_n\lesssim s_n^{-b}$ for any $b\in(\mathrm{dim}_A(X),a)$ and to be nested. Therefore, in the case of stochastic coupling we can pick the rate $r_n=s_n=q_n^{-1/(a+2)}$,
which can readily be seen to satisfy the conditions $\sqrt{K_n/q_n}\in o(r_n)$ and $\sqrt{\log\log(q_n)/q_n}\in o(r_n)$ from \cref{thm:stochConv}.
As for the nestedness, for any $\tilde b\in(\mathrm{dim}_A(X),b)$ let $C\geq1$ be such that any $r$-ball can be covered by no more than $C(r/s)^{\tilde b}$ $s$-balls and pick $\ell=C^{1/(b-\tilde b)}$.
Define $s_0=\ell\mathrm{diam}(X)$ as well as the indices and radii
\begin{equation*}
l(n)=\min\{l\in\Z\,|\,2\ell^{-l}\leq s_n\}
\quad\text{and}\quad
R_n=\ell^{-l(n)}
\quad\text{for }n\in\N_0,
\end{equation*}
thus we have $s_n/\ell\leq2R_n\leq s_n$ for all $n$.
We next inductively define a sequence of coverings $(B_n^k)_{n,k=1,\ldots,K_n}$ of $X$ by balls of radii $R_n$ as follows:
$B_0^1=X$ is a covering of $X$ consisting of a single ball of radius $R_0$, thus $K_0=1$.
Given $(B_n^k)_k$, if $R_{n+1}=R_n$, then we define $(B_{n+1}^k)_k=(B_n^k)_k$ so that $K_{n+1}/K_n=1$.
Otherwise we define $(B_{n+1}^k)_k$ as the union of coverings of each ball $B_n^k$ by at most $C(R_n/R_{n+1})^{\tilde b}$ balls of radius $R_{n+1}$
so that $K_{n+1}/K_n\leq C(R_n/R_{n+1})^{\tilde b}=C(\ell^{l(n+1)-l(n)})^{\tilde b}\leq C^{l(n+1)-l(n)}(\ell^{l(n+1)-l(n)})^{\tilde b}$.
Obviously, in both cases
\begin{equation*}
K_{n+1}/K_n\leq C^{l(n+1)-l(n)}(\ell^{l(n+1)-l(n)})^{\tilde b}.
\end{equation*}
Therefore, by construction, the covering $(B_n^k)_k$ contains at most
\begin{equation*}
K_n
=K_0\cdot\tfrac{K_1}{K_0}\cdot\tfrac{K_2}{K_1}\cdot\ldots\cdot\tfrac{K_n}{K_{n-1}}
\leq C^{l(n)-l(0)}(\ell^{l(n)-l(0)})^{\tilde b}
=\ell^{(l(n)-l(0))b}
=(\tfrac{R_0}{R_n})^b
\lesssim s_n^{-b}
\end{equation*}
balls of diameter no larger than $s_n$.
The partitions $(C_n^k)_{k,n}$ are inductively created from the ball coverings such that $C_n^k\subset B_n^k$ for all $n,k$:
We first set $C_0^1=B_0^1$.
Given $(C_n^k)_{k}$, if $R_{n+1}=R_n$, then we define $(C_{n+1}^k)_k=(C_n^k)_k$.
Otherwise, if $B_{n+1}^{k_1},\ldots,B_{n+1}^{k_j}$ are the covering of $B_n^k\supset C_n^k$, then we set
\begin{equation*}
C_{n+1}^{k_i}=\left(B_{n+1}^{k_i}\cap C_n^k\right)\setminus\bigcup_{l=1}^{i-1}C_{n+1}^{k_l},
\quad i=1,\ldots,j.
\end{equation*}
Thus, by construction the partitions $(C_n^k)_{k,n}$ are nested and have $K_n\lesssim s_n^{-b}$ subsets as desired.

Finally, for a bounded sequence of measures flat implies weak-* convergence.
\end{proof}

\section{Benamou--Brenier regularization}\label{sec:BenamouBrenier}
Here we analyze the convergence properties of the Benamou--Brenier regularization \eqref{eqn:BenamouBrenier}.
The latter fits well to dynamic inverse problems in which a sought mass distribution (here the radionuclide distribution) changes over time while the total mass is conserved.

The involved continuity equation \eqref{eqn:CE_introduction} is to be understood in the distributional sense, i.e.\
\begin{equation}\label{eqn:defCE}
\int_{[0,T]\times\dom}\partial_t\varphi \dint \density + \int_{[0,T]\times\dom}\dotProd{\nabla_x\varphi}{\dint \momentum} = 0
\qquad\text{for all }\varphi\in C_c^1((0,T)\times\dom).
\end{equation}
Note that by \cite[Lemma 1.1.2]{chizat_unbalancedOT}, since $\dom$ is compact, any $\density$ satisfying \eqref{eqn:defCE} lies in
\begin{equation}\label{eqn:constantMass}
\meas_c=\{ \density\in\measp(\timeInterval\!\times\!\dom)\ | \ \density\text{ disintegrates to } \density=\d t\otimes\densityT{t}, \ \densityT{t}(\dom)\text{ independent of }t\text{ a.e.} \}
\end{equation}
(i.e.\ we have mass conservation in time).
The Benamou--Brenier regularization \eqref{eqn:BenamouBrenier} for $\density\in\measp(\spacetime)$ and $\momentum\in\meas([0,T]\times\R^3)^3$ can now be more precisely expressed as
\begin{equation}
\BBEnergy(\density,\momentum)=
\begin{cases}
\int_0^T\int_\dom\left(\frac{\d \momentum_t}{\d \densityT{t}}\right)^2\dint\densityT{t}\dint t&\text{if }\density\ge 0, \ \momentum\ll\density \text{ and \eqref{eqn:defCE} holds,}\\
\infty&\text{else.}
\end{cases}
\end{equation}
\begin{remark}[Properties of $\BBEnergy$]\label{rem:propertiesBenamouBrenier}
The following properties can e.g.\ be found in \cite{BenamouBrenier,OTAppliedMath,chizat_unbalancedOT}.

\begin{enumerate}[label=(\alph*)]
	\item\label{enm:Slsc}
	$\BBEnergy$ is nonnegative, convex, and lower semi-continuous w.r.t.\ weak-* convergence.
	\item\label{enm:massConservation}
	If $\BBEnergy(\density,\momentum)<\infty$, then $\density\ge 0$, $\momentum\ll\density$ \cite[Prop.\,5.18]{OTAppliedMath}, and $\density\in\meas_c$ \cite[Lemma 1.1.2]{chizat_unbalancedOT}.
	Furthermore, $t\mapsto\densityT{t}$ is weakly-* continuous \cite[Prop.\,1.1.3]{chizat_unbalancedOT}.
	\item
  $\BBEnergy$ can be used to compute the Wasserstein-2 distance $\mathbb W_2(\mu,\nu)$ between two nonnegative measures $\mu$ and $\nu$ of same mass by
	\begin{align*}
	\mathbb W_2^2(\mu,\nu)
	&=\min\left\{ \int_0^1 \norm{ \frac{\d \momentum_t}{\d \densityT{t}}}_{\mathrm{L}^2(\densityT{t})}^2\d  t \ \middle| \ \partial_t\density+\mdiv\,\momentum=0, \ \density_0=\mu, \ \density_1=\nu \right\}\\
	&=\min\left\{T\BBEnergy(\density,\momentum)\ \middle|\ \density\in\measp(\spacetime),\momentum\in\meas([0,T]\times\R^3)^3,\density_0=\mu, \density_T=\nu\right\}.
	\end{align*}
	\item For every absolutely continuous curve $(\densityT{t})_{t\in[0,T]}$ in the Wasserstein-2 space (the set of constant mass measures with metric $\mathbb{W}_2$) there exists a vector field $v_t\in \mathrm{L}^2(\densityT{t})^3$ such that $\partial_t\densityT{t}+\mdiv\,(v_t\densityT{t})=0$ and $\BBEnergy(\density,\momentum)=\int\normS{v_t}^2_{\mathrm{L}^2(\densityT{t})}\dint t<\infty$ \cite[Thm.\,5.14]{OTAppliedMath} for $\momentum=v\density$.
\end{enumerate}
\end{remark}
A direct consequence is the lower semi-continuity after minimizing for the momentum.
\begin{lemma}[Lower semi-continuity of $\BBEnergy$]\label{thm:continuityBBEnergy}
The function $\density\mapsto\min_\momentum\BBEnergy(\density,\momentum)\in[0,\infty]$ is well-defined, convex and weakly-* lower semi-continuous.
\end{lemma}
\begin{proof}
The convexity is a standard consequence of $\BBEnergy$ being convex.
Now take a sequence $\densityN{n}\convStar\density$ and let $\momentumN{n}$ such that $\BBEnergy(\densityN{n},\momentumN{n})\leq\inf_\momentum\BBEnergy(\densityN{n},\momentum)+\frac1n$.
Without loss of generality we may assume $\liminf_{n\to\infty}\inf_\momentum\BBEnergy(\densityN{n},\momentum)=\lim_{n\to\infty}\inf_\momentum\BBEnergy(\density_{n},\momentum)=\lim_{n\to\infty}\BBEnergy(\density_{n},\momentum_n)$
(else we can pass to a subsequence)
as well as $\lim_{n\to\infty}\BBEnergy(\densityN{n},\momentum_n)<\infty$ (else there is nothing to show). Jensen's inequality implies
\begin{equation*}
\label{EQGammConvNormMomentum}
\norm{\momentumN{{n}}}=\norm{\tfrac{\d \momentumN{{n}}}{\d \densityN{{n}}}}_{\mathrm{L}^1(\densityN{{n}})}\le\norm{\densityN{{n}}}^{\frac{1}{2}}\BBEnergy(\densityN{{n}},\momentumN{{n}})^{\frac{1}{2}}\lesssim1
\end{equation*}
so that we have a uniform bound on $\norm{\momentumN{{n}}}$ and can thus extract a weakly-* converging subsequence (still indexed by $n$) $\momentumN{{n}}\convStar\momentum_\infty$. By weak-* lower semi-continuity of $\BBEnergy$ (\cref{rem:propertiesBenamouBrenier}\ref{enm:Slsc}),
\begin{equation*}
\liminf_n\inf_\momentum\BBEnergy(\densityN{n},\momentum)=\lim_{n}\BBEnergy(\densityN{{n}},\momentumN{{n}})\ge\BBEnergy(\density,\momentum_\infty)\ge\inf_\momentum\BBEnergy(\density,\momentum),
\end{equation*}
thus the function is weakly-* lower semi-continuous.
Repeating the above derivations for $\density_n=\density$ shows existence of minimizers (allowing the minimal value to be infinite) so that the function is well-defined.
\end{proof}

Also note for later use that $\meas_c$ is weakly-* closed.
\begin{lemma}[Weak-* closedness of mass conservation]\label{thm:closednessOfBSet}
$\meas_c$
is closed under weak-* convergence, and $\norm{\densityN{n}}\to\norm\density$ for any $\densityN{n}\weakstarto\density$ in $\meas_c$.
\end{lemma}
\begin{proof}
Let $\densityN{n}\convStar\density$ in $\measp(\spacetime)$ with $\densityN{n}\in\meas_c$ for all $n$.
Consider the projection $\pi_1\colon\spacetime\rightarrow[0,T]$, $(t,x)\mapsto t$.
We have $\pushforward{(\pi_1)}{\densityN{n}}=c_n\mathcal{L}^1\restr[0,T]$ for some constant $c_n=\densityTN{t}{n}(\dom)\geq0$.
Since $\pi_1$ is continuous, $\pushforward{(\pi_1)}{\densityN{n}}\convStar\pushforward{(\pi_1)}\density$
and therefore $\pushforward{(\pi_1)}\density=c\mathcal{L}^1\restr[0,T]$ for $c=\lim_{n\to\infty}c_n\geq0$.
By the disintegration theorem this implies $\density=\d t\otimes\densityT{t}$ and $\densityT{t}(\dom)=c$.
\end{proof}

The following lemma is required in our variational analysis of the reconstruction functional to define the recovery sequence of the $\Gamma$-convergence.
More specifically, it is necessary in case the ground truth has (infinite) Benamou--Brenier energy since one wants the ground truth to be recovered in the limit despite the regularization discouraging it. 
\begin{lemma}[Bounded growth of $\BBEnergy$]\label{thm:RecoverySequence}
Let $\density=\d t\otimes \densityT{t}\in\measp(\spacetime)$ and $\densityT{t}(\dom)=\text{const.}$ almost everywhere. For any sequence $\delta_n\to0$ we can find a sequence of Radon measures $(\densityN{n},\momentumN{n})_n\subset\measp(\spacetime)\times\meas([0,T]\times\R^3)^3$ such that $\densityN{n}\convStar\density$ and $\BBEnergy(\densityN{n},\momentumN{n})\leq1/\delta_n$, and we can additionally choose one of the following assertions to hold:
\begin{enumerate}[label=(\alph*)]
\item
$\mathbb{W}_2(\densityN{n},\density)\lesssim(\delta_n)^{1/6}$.
\notinclude{ and
\begin{align*}
\left((t,x)\mapsto\int_\dom f(x,y)\dint\densityTN{t}{n}(y)\right)\xrightarrow{\mathrm{L}^2(\spacetime)}\left((t,x)\mapsto\int_\dom f(x,y)\dint\densityT{t}(y)\right)
\end{align*}}

\item
\label{enm:RecoverySequenceB} Given a nonnegative null sequence $\Delta T_n\to0$ there exists a subsequence along which $\mathbb W_2(\densityTN{t_n}{n},\densityT{t})\to0$
for almost every $t$ and any sequence $t_n\in[0,T]$ with $\absNorm{t_n-t}\le\Delta T_n$.
\end{enumerate}
\end{lemma}
\begin{proof}
We mainly adapt the constructions from \cite[Thm.\,5.14]{OTAppliedMath} to our setting. We break down the proof into multiple steps.

\paragraph{Step 1: Approximating $\density$ by regular measures $\densityN{n}$ and constructing corresponding momenta $\momentumN{n}$.}
The curve $t\mapsto\densityT{t}$ might not be continuous w.r.t.\ the weak-* convergence and hence might not admit any $\momentum$ with $\BBEnergy(\density,\momentum)<\infty$. Therefore we approximate $\density$ by more regular measures that allow finite $\BBEnergy$.

\paragraph{Step 1a: Smoothing in time.}
Define the temporal mollifier
\begin{equation*}
\mollifierTime(t)=
\begin{cases}
c\exp\left(\frac{1}{t^2-1}\right) \ &\text{if }\absNorm{t}<1,\\
0 \ &\text{if }\absNorm{t}\ge 1 
\end{cases}
\quad\text{with }c\text{ such that } \ \ \int_\R\mollifierTime(t)\dint t=1,
\end{equation*}
and set $\mollifierTime_\varepsilon(t)=\frac{1}{\varepsilon}\mollifierTime\left(\frac{t}{\varepsilon}\right)$.
We then extend $\density$ onto the time interval $[-T,2T]$ by reflection in time,
\begin{equation*}
\densityT{t}=\density_{R(t)}
\qquad\text{with }
R(t)=-t
\text{ for $t<0$, }
R(t)=2T-t
\text{ for $t>T$, and }
R(t)=t
\text{ else,}
\end{equation*}
and define $\density_\varepsilon=\d t\otimes\density_{\varepsilon, t}\in\measp(\spacetime)$ as the temporal convolution of $\density$ with $\mollifierTime_\varepsilon$,
\begin{equation*}
\density_{\varepsilon,t}=
\int_{-\varepsilon}^{T+\varepsilon}\mollifierTime_\varepsilon(t-s)\density_s\dint s=(\mollifierTime_\varepsilon\ast\density_s)(t).
\end{equation*}
This curve is Hölder continuous in the Wasserstein-2 space over $\dom$.
Indeed, letting $L(\varepsilon)\simeq1/\varepsilon^2$ be the Lipschitz constant of $\mollifierTime_\varepsilon$ and abbreviating partitions of $\dom$ by $P$, for $a,b\in[0,T]$ we have
\begin{multline}\label{eqn:recoverySeqW1}
\mathbb{W}_2^2(\density_{\varepsilon,a},\density_{\varepsilon,b})
\lesssim\mathrm{diam}(\dom)^2\norm{\density_{\varepsilon,a}-\density_{\varepsilon,b}}
\simeq\sup_{P}\sum_{A\in P}\absNorm{\density_{\varepsilon, a}(A)-\density_{\varepsilon,b}(A)}\\
\le\sup_{P}\sum_{A\in P}\int_{-\varepsilon}^{T+\varepsilon}\absNorm{\mollifierTime_{\varepsilon}(a-s)-\mollifierTime_\varepsilon(b-s)}\density_s(A)\dint s
\lesssim\absNorm{a-b}\norm{\density}L(\varepsilon).
\end{multline}

\paragraph{Step 1b: Construct velocity fields to satisfy continuity equation.}
Now that we have obtained a time continuous curve, the next step is to construct associated velocities to satisfy the continuity equation. We use the construction from \cite[Thm.\,5.14]{OTAppliedMath} to prove the existence of a measure $\overline{\density}_\varepsilon^{k}$ and a function $\overline{v}^k_{\varepsilon}\in\mathrm{L}^2\left(\overline{\density}^k_{\varepsilon}\right)^3$ such that
\begin{equation*}
\partial_t\overline{\density}^k_\varepsilon+\mdiv(\overline{v}^k_\varepsilon\overline{\density}^k_\varepsilon)=0
\quad\text{(and, as seen in the subsequent step,}\quad
\normS{\overline{v}_{\varepsilon, t}^k}^2_{\mathrm{L}^2(\overline{\density}_{\varepsilon, t}^k)}\lesssim\tfrac{k}{\varepsilon^2}\text{).}
\end{equation*}
For the construction, let $k\in\mathbb{N}$. We set $\overline{\density}_{\varepsilon, Ti/k}^k=\mollifier_k\ast\density_{\varepsilon,Ti/k}$ for $i=0,\ldots,k$ with $\mollifier_k$ an even mollifier supported on $B_{1/k}(0)\subset\R^3$.
The support of $\overline{\density}_{\varepsilon,iT/k}^k$ is therefore contained in the compact and convex set $\dom_k\coloneqq\{ x\in\R^3\ | \ \dist(\dom,x)\le 1/k \}\subset\hat\dom\coloneqq D_1$. By absolute continuity of $\overline{\density}_{\varepsilon,iT/k}^k$ w.r.t.\ the Lebesgue measure, the Wasserstein-2 distance from $\overline{\density}_{\varepsilon,Ti/k}^k$ to $\overline{\density}_{\varepsilon,T(i+1)/k}^k$ induces optimal transport maps
\begin{align*}
T^{i,k}\colon\hat\dom\to\hat\dom
\end{align*}
for $i\in\{ 0,\ldots,k-1 \}$. For $t\in(\frac{Ti}{k},\frac{T(i+1)}{k})$ we interpolate $T^{i,k}_t\coloneqq\frac{1}{T}(((i+1)T-kt)\text{id}+(kt-iT)T^{i,k})$ and set
\begin{align*}
\overline{\density}_{\varepsilon,t}^k\coloneqq\pushforward{(T_t^{i,k})}{\overline{\density}_{\varepsilon,Ti/k}^k},
\end{align*}
which describes a linear particle motion between any position $x$ at time $Ti/k$ and position $T^{i,k}(x)$ at time $T(i+1)/k$. The associated velocity reads
\begin{align*}
\overline{v}_{\varepsilon}^{i,k}\coloneqq \frac{k}{T}(T^{i,k}-\text{id}).
\end{align*}
Since $T_t^{i,k}$ is injective for every $t\in(\frac{Ti}{k},\frac{T(i+1)}{k})$ \cite[Lemma 4.23]{OTAppliedMath} we finally define
\begin{align*}
\overline{v}^{k}_{\varepsilon,t}\coloneqq\overline{v}_{\varepsilon}^{i,k}\circ \left(T_t^{i,k}\right)^{-1}.
\end{align*}
This construction yields a pair of measures $(\overline{\density}_{\varepsilon}^k,\overline{v}^k_\varepsilon\overline{\density}_{\varepsilon}^k)\eqqcolon(\overline\density^k_\varepsilon,\overline\momentum^k_\varepsilon)$ satisfying the continuity equation. To see this, consider for $\psi\in C^1(\hat\dom)$ and $t\in\left( \frac{Ti}{k},\frac{T(i+1)}{k} \right)$
\begin{multline*}
\frac{\d }{\d t}\int_{\hat\dom}\psi(x)\dint\overline{\density}^k_{\varepsilon,t}(x)
=\frac{\d }{\d t}\int_{\hat\dom}\psi(T_t^{i,k}(x))\dint\overline{\density}^k_{\varepsilon,Ti/k}(x)\\
=\int_{\hat\dom}\dotProd{\nabla_x\psi(T_t^{i,k}(x))}{\overline{v}_{\varepsilon}^{i,k}(x)}\dint\overline{\density}^k_{\varepsilon,Ti/k}(x)
=\int_{\hat\dom}\dotProd{\nabla_x\psi(x)}{\overline{v}_{\varepsilon,t}^k(x)}\dint\overline{\density}^k_{\varepsilon,t}(x),
\end{multline*}
where the order of integration and differentiation may be changed due to the boundedness of $t\mapsto \dotProd{\nabla_x\psi(T_t^{i,k}(x))}{\overline{v}_{\varepsilon}^{i,k}(x)}$ and the finiteness of $\overline{\density}_{\varepsilon,Ti/k}^k$. This means that $(\overline{\density}_{\varepsilon}^k,\overline{v}^k_\varepsilon\overline{\density}_{\varepsilon}^k)$ satisfies the continuity equation in the weak sense which is equivalent to \eqref{eqn:defCE} \cite[Prop.\,4.2]{OTAppliedMath}.

\paragraph{Step 1c: Bounding the norm of the velocity.}
For $t\in(\frac{Ti}{k},\frac{T(i+1)}{k})$ we also get
\begin{multline}\label{EQRecoverySeqBoundV}
\normS{\overline{v}_{\varepsilon,t}^k}^2_{\mathrm{L}^2(\overline{\density}^k_{\varepsilon,t})}
=\int_{\hat\dom}\absNormS{\overline{v}^k_{\varepsilon,t}}^2\dint\overline{\density}_{\varepsilon,t}^k=\int_{\hat\dom}\absNormS{\overline{v}_\varepsilon^{i,k}}^2\circ\left( T_t^{i,k} \right)^{-1}\dint\pushforward{(T_t^{i,k})}{\overline{\density}_{\varepsilon,Ti/k}^k}
=\int_{\hat\dom}\absNormS{\overline{v}_{\varepsilon}^{i,k}}^2\dint\overline{\density}_{\varepsilon,Ti/k}^k\\
=\frac{k^2}{T^2}\mathbb{W}_2^2(\overline{\density}_{\varepsilon,Ti/k}^k,\overline{\density}_{\varepsilon,T(i+1)/k}^k)
\le\frac{k^2}{T^2}\mathbb{W}_2^2(\density_{\varepsilon,Ti/k},\density_{\varepsilon,T(i+1)/k})
\lesssim k\frac{1}{\varepsilon^2},
\end{multline}
where we used \cite[Lemma 5.2]{OTAppliedMath} and \eqref{eqn:recoverySeqW1} in the last two steps.

\paragraph{Step 1d: Getting the right support for the approximating measures.}
As a final adjustment we need to reduce the potentially too large support of $\overline{\density}_{\varepsilon}^k$ from $\timeInterval\times\hat\dom$ to $\spacetime$.
For our convex domain $\dom$ this is easiest via spatial rescaling.
Hence let $m\coloneqq\min_{x\in\partial\dom}\absNorm{x}>0$ (recall that we assumed $0\in\mathrm{int}(\dom)$) and set $c_k\coloneqq\frac{1}{1+1/(mk)}$, then $c_k\dom_k\subset\dom$:
Indeed, it suffices to show $c_k\partial\dom_k\subset\dom$, so let $x\in\partial\dom_k$ and let $z$ be its orthogonal projection onto $\dom$,
then $q=c_kx\in\dom$ since $q=(1-c_k)y+c_kz$ is a convex combination of the points $z$ and $y=z+\frac{q-z}{1-c_k}$ in $\dom$
(that $y\in\dom$ can be seen from $\absNorm{y}=\absNorm{\frac{c_k}{1-c_k}(x-z)}=mk\absNorm{x-z}\leq m$).
We then define the desired approximating measures via their time slices,
\begin{equation*}
\density_{\varepsilon,t}^{k}=\pushforward{(x\mapsto c_{k}x)}{(\overline{\density}_{\varepsilon,t}^{k})}\quad\text{and}\quad\momentum_{\varepsilon,t}^{k}=\pushforward{(x\mapsto c_{k}x)}{(c_k\overline{v}^{k}_{\varepsilon,t}\overline{\density}^{k}_{\varepsilon,t})}
\end{equation*}
for every $t$. Note, that the factor $c_k$ in $(c_k\overline{v}^{k}_{\varepsilon,t}\overline{\density}^{k}_{\varepsilon,t})$ is necessary in order for the final measures to satisfy the continity equation \cite[Prop.\,1.1.6]{chizat_unbalancedOT}.\\
Finally we set $\densityN{n}=\density_{\varepsilon_n}^{k_n}$ and $\momentumN{n}=\momentum_{\varepsilon_n}^{k_n}$ for sequences $\varepsilon_n\to 0$ and $k_n\to\infty$.
It remains to specify those sequences such that the constructed measures possess the desired properties.

\paragraph{Step 2: Computing a bound on $\BBEnergy(\densityN{n},\momentumN{n})$.}
The Radon--Nikodym derivative of $\momentumTN{t}{n}$ w.r.t.\ $\densityTN{t}{n}$ is readily found as $x\mapsto c_{k_n}\overline{v}^{k_n}_{\varepsilon_n,t}(x/c_{k_n})$.
Using this in \eqref{EQRecoverySeqBoundV} we get
\begin{equation*}
\BBEnergy(\densityN{n},\momentumN{n})=\int_0^T\int_\dom\left( \frac{\d \momentum_{n,t}}{\d \densityTN{t}{n}} \right)^2\d \densityTN{t}{n}\dint t
=c_{k_n}^2\int_0^T\int_{\hat\dom}|\overline{v}^{k_n}_{\varepsilon_n,t}|^2\d \overline{\density}^{k_n}_{\varepsilon_n,t}\dint t
\lesssim \frac{{k_n}}{\varepsilon_n^2}.
\end{equation*}


\paragraph{Step 3: Intermediate estimates of $\mathbb{W}_2(\densityN{n},\density)$.}
Next we show $\mathbb{W}_2(\densityN{n},\density)\to0$ (even at rate $\delta_n^{1/6}$ if desired), which is known to imply $\densityN{n}\convStar\density$.
The triangle inequality yields
\begin{equation*}
\mathbb{W}_2(\densityN{n},\density)
\leq\mathbb{W}_2(\densityN{n},\overline{\density}^{k_n}_{\varepsilon_n})+\mathbb{W}_2(\overline{\density}^{k_n}_{\varepsilon_n},\density_{\varepsilon_n})+\mathbb{W}_2(\density_{\varepsilon_n},\density).
\end{equation*}


\paragraph{Step 3a: Estimate of $\mathbb{W}_2(\densityN{n},\overline{\density}^{k_n}_{\varepsilon_n})$.}
Since $\densityTN{t}{n}=\density^{k_n}_{\varepsilon_n,t}$ is defined as a pushforward of $\overline{\density}^{k_n}_{\varepsilon_n,t}$,
it is straightforward to provide the transport plan $\transportPlan^{k_n}=\pushforward{(x\mapsto(x,c_{k_n}x))}{(\overline{\density}_{\varepsilon_n,t}^{k_n})}$ from the latter to the former.
With this we obtain
\begin{equation*}
\mathbb{W}_2(\densityTN{t}{n},\overline{\density}^{k_n}_{\varepsilon_n,t})^2
\leq\int_{\hat\dom\times\hat\dom}\absNorm{x-y}^2\dint\transportPlan^{k_n}(x,y)
=\int_{\hat\dom}\absNorm{x-c_{k_n}x}^2\dint\overline{\density}_{\varepsilon_n,t}^{k_n}
\lesssim\sup_{x\in\hat\dom}\absNorm{x-c_{k_n}x}^2
\lesssim\frac1{k_n^2},
\end{equation*}
which implies $\mathbb{W}_2(\densityN{n},\overline{\density}^{k_n}_{\varepsilon_n})\leq\int_0^T\mathbb{W}_2(\densityTN{t}{n},\overline{\density}^{k_n}_{\varepsilon_n,t})\dint t\lesssim1/k_n$.


\paragraph{Step 3b: Estimate of $\mathbb{W}_2(\overline{\density}^{k_n}_{\varepsilon_n},\density_{\varepsilon_n})$.}
Similar to \cite[proof of Thm.\,5.14]{OTAppliedMath} we get
\begin{equation*}\label{eqn:BBEnergyProofWassersteinEstimate}
\mathbb{W}_2(\overline{\density}_{\varepsilon_n,t}^{k_n},\density_{\varepsilon_n,t})\le \mathbb{W}_2(\overline{\density}_{\varepsilon_n,t}^{k_n},\overline{\density}_{\varepsilon_n,Ti/k_n}^{k_n})+\mathbb{W}_2(\overline{\density}_{\varepsilon_n,Ti/k_n}^{k_n},\density_{\varepsilon_n,Ti/{k_n}})+\mathbb{W}_2(\density_{\varepsilon_n,Ti/{k_n}},\density_{\varepsilon_n,t})
\end{equation*}
for $t\in(\frac{Ti}{{k_n}},\frac{T(i+1)}{{k_n}})$.
The construction $\overline{\density}^{k_n}_{\varepsilon_n,t}=\pushforward{(T_t^{i,{k_n}})}{\overline{\density}^{k_n}_{\varepsilon_n,Ti/{k_n}}}$ as a pushforward yields
\begin{multline*}
\mathbb{W}^2_2(\overline{\density}_{\varepsilon_n,t}^{k_n},\overline{\density}_{\varepsilon_n,Ti/k_n}^{k_n})
\le\int_{\hat\dom}\absNormS{T_t^{i,{k_n}}(x)-x}^2\dint\overline{\density}_{\varepsilon_n,Ti/k_n}^{k_n}\\
\le\int_{\hat\dom}\absNormS{T^{i,{k_n}}(x)-x}^2\dint\overline{\density}_{\varepsilon_n,Ti/k_n}^{k_n}
=\frac{T^2}{{k_n}^2}\normS{\overline{v}_{\varepsilon_n}^{i,{k_n}}}^2_{\mathrm{L}^2(\overline{\density}^{k_n}_{\varepsilon_n,Ti/k_n})}
\lesssim\frac{1} {{k_n}\varepsilon_n^2},
\end{multline*}
where we used \eqref{EQRecoverySeqBoundV} in the last step. For the second summand we modify the proof of \cite[Lemma 5.2]{OTAppliedMath}.
With the mollifiers $\mollifier_{k_n}$ from above we define a transport plan $\transportPlan^{k_n}$ from $\density_{\varepsilon_n,Ti/{k_n}}$ to $\overline{\density}_{\varepsilon_n,Ti/{k_n}}^{k_n}$ by
\begin{equation*}
\int_{\hat\dom\times\hat\dom}\varphi\dint\transportPlan^{k_n}
\coloneqq\int_{\dom}\int_{B_{1/{k_n}}(0)}\varphi(x,x+z)\mollifier_{k_n}(z)\dint z\dint\density_{\varepsilon_n,Ti/{k_n}}(x)
\qquad\text{for all }\varphi\in C(\hat\dom\times\hat\dom).
\end{equation*}
This yields the estimate
\begin{equation*}
\mathbb{W}^2_2(\overline{\density}_{\varepsilon_n,Ti/{k_n}}^{k_n},\density_{\varepsilon_n,Ti/{k_n}})
\le\int_{\hat\dom\times\hat\dom}\absNorm{x-y}^2\dint\transportPlan^{k_n}(x,y)
=\int_{\dom}\int_{B_{1/{k_n}}(0)}\absNorm{z}^2\mollifier_{k_n}(z)\dint z\dint \density_{\varepsilon_n,Ti/{k_n}}
\lesssim\frac{1}{k_n^2}.
\end{equation*}
The last summand can be estimated via \eqref{eqn:recoverySeqW1} by $\frac{1}{\sqrt{{k_n}}\varepsilon_n}$. Overall, assuming $1\lesssim k_n\varepsilon_n^2$ (to be ensured later), we obtain
\begin{align*}
\mathbb{W}_2(\overline{\density}_{\varepsilon_n,t}^{k_n},\density_{\varepsilon_n,t})\lesssim\frac{1}{\sqrt{{k_n}}\varepsilon_n}
\end{align*}
for every $t\in[0,T]$ and thus the same estimate for $\mathbb{W}_2(\overline{\density}^{k_n}_{\varepsilon_n},\density_{\varepsilon_n})$.

\paragraph{Step 3c: Estimate of $\mathbb{W}_2(\density_{\varepsilon_n},\density)$.}
Since $\density_{\varepsilon_n}$ derives from $\density$ via a convolution, we can readily construct a transport plan $\transportPlan^{\varepsilon_n}$ from the latter to the former
(using the mollifier $\mollifierTime_{\varepsilon_n}$ and the time reflection $R$ from above),
\begin{equation*}
\int_{\spacetime\times\spacetime}\varphi\dint\transportPlan^{\varepsilon_n}
\coloneqq\int_{\spacetime}\int_{-\varepsilon_n}^{\varepsilon_n}\mollifierTime_{\varepsilon_n}(s)\varphi(t,x,R(t+s),x)\dint s\dint\densityT{t}(x)\dint t.
\end{equation*}
Indeed, the marginals of $\transportPlan^{\varepsilon_n}$ are $\density$ and $\density_{\varepsilon_n}$, where the latter can be seen from
\begin{multline*}
\int_{\spacetime\times\spacetime}\phi(s,y)\dint\transportPlan^{\varepsilon_n}(t,x,s,y)
=\int_{\spacetime}\int_{-\varepsilon_n}^{\varepsilon_n}\mollifierTime_{\varepsilon_n}(s)\phi(R(t+s),x)\dint s\dint\densityT{t}(x)\dint t\\
=\int_{\spacetime}\int_{-\varepsilon_n}^{\varepsilon_n}\mollifierTime_{\varepsilon_n}(s)\phi(t,x)\dint s\dint\densityT{R(t-s)}(x)\dint t
=\int_{\spacetime}\phi(t,x)\dint\density_{\varepsilon_n, t}(x)\dint t
\end{multline*}
(the second equality exploits the evenness of $\mollifierTime_{\varepsilon_n}$ and $\int_0^Tg(R(t-s))f(t)\dint t+\int_0^Tg(R(t+s))f(t)\dint t=\int_0^Tg(t)f(R(t-s))\dint t+\int_0^Tg(t)f(R(t+s))\dint t$ for any functions $f,g$,
which can readily be checked by splitting $[0,T]$ into the subintervals $[0,s]$, $[s,T-s]$, $[T-s,T]$).
It follows
\begin{multline*}
\mathbb{W}^2_2(\density_{\varepsilon_n},\density)
\leq\int_{\spacetime\times\spacetime}\absNorm{(t,x)-(s,y)}^2\dint\transportPlan^{\varepsilon_n}(t,x,s,y)\\
\leq\int_{\spacetime}\int_{-\varepsilon_n}^{\varepsilon_n}\mollifierTime_{\varepsilon_n}(s)\absNorm{s}^2\dint s\dint\densityT{t}(x)\dint t
\leq\norm{\density}\varepsilon_n^2.
\end{multline*}

\paragraph{Step 3d: Convergence rate for $\mathbb{W}_2(\densityN{n},\density)$.}
Combining the previous steps we obtain $\mathbb{W}_2(\densityN{n},\density)\lesssim\frac{1}{\sqrt{k_n}\varepsilon_n}+\varepsilon_n$,
which converges to zero as long as $1/k_n\in o(\varepsilon_n^2)$.
%
%
Since from $\BBEnergy(\densityN{n},\momentumN{n})\lesssim k_n/\varepsilon_n^2$ we require $\varepsilon_n\gtrsim\sqrt{k_n\delta_n}$ to guarantee $\BBEnergy(\densityN{n},\momentumN{n})\leq1/\delta_n$,
the optimal parameter choice can be found as $\varepsilon_n\simeq\sqrt{k_n\delta_n}$ and $k_n=\ceil{\delta_n^{-2/3}}$, which yields the rate $\mathbb{W}_2(\densityN{n},\density)\lesssim\delta_n^{1/6}$.

\paragraph{Step 4: Convergence of $\mathbb W_2(\densityTN{t_n}{n},\densityT{t})$.}
By the triangle inequality we have
\begin{multline*}
\mathbb W_2(\densityTN{t_n}{n},\densityT{t})
\leq\mathbb W_2(\densityTN{t_n}{n},\overline{\density}^{k_n}_{\varepsilon_n,t_n})
+\mathbb W_2(\overline{\density}^{k_n}_{\varepsilon_n,t_n},\density_{\varepsilon_n,t_n})
+\mathbb W_2(\density_{\varepsilon_n,t_n},\density_{\varepsilon_n,t})
+\mathbb W_2(\density_{\varepsilon_n,t},\densityT{t})\\
\lesssim\frac1{k_n}+\frac1{\sqrt{k_n}\varepsilon_n}+\frac{\sqrt{\Delta T_n}}{\varepsilon_n}
+\mathbb W_2(\density_{\varepsilon_n,t},\densityT{t}),
\end{multline*}
where the estimates of the first three summands stem from steps 3a and 3b as well as \eqref{eqn:recoverySeqW1}.
Hence, choosing for instance $k_n=\ceil{\delta_n^{-2/3}}$ and $\varepsilon_n=\max\{\Delta T_n^{1/4},\sqrt{k_n\delta_n}\}$
(for which from step 3 we have $\densityN{n}\convStar\density$ and $\BBEnergy(\densityN{n},\momentumN{n})\leq1/\delta_n$), the first three summands vanish in the limit,
and it remains to show $\mathbb W_2(\density_{\varepsilon_n,t},\densityT{t})\to0$ along a subsequence for almost every $t$.
Since $\mathbb W_2$ metrizes weak-* convergence, it actually suffices to show, along a subsequence and for almost every $t$, $\density_{\varepsilon_n,t}\convStar\densityT{t}$ as $n\to\infty$
or equivalently $\int_{\dom}\varphi\dint(\density_{\varepsilon_n,t}-\densityT{t})\to0$ for all $\varphi$ from a countable dense subset $\mathfrak C\subset C(\dom)$.
Now let $\mathfrak C=\{\varphi_1,\varphi_2,\ldots\}$ and define, for $i\in\N$,
\begin{equation*}
g_i\in\mathrm{L}^\infty(\R)\cap\mathrm{L}^1(\R),\qquad
g_i(t)=\int_{\dom}\varphi_i\dint\densityT{t}
\quad\text{ for }t\in[-T,2T]\text{ and }g_i(t)=0\text{ else}
\end{equation*}
(recall that $\densityT{t}$ was extended to $t\notin[0,T]$ by time reflection).
Then $\int_{\dom}\varphi_i\dint\density_{\varepsilon_n,t}=g_i*\mollifierTime_{\varepsilon_n}(t)$ by definition of $\density_{\varepsilon_n}$. We will inductively construct a Lebesgue-nullset $N\subset[0,T]$ and a subsequence of $(\varepsilon_n)_n$
such that along this subsequence, for all $i\in\N$ we have $g_i*\mollifierTime_{\varepsilon_n}\convN g_i$ pointwise on $[0,T]\setminus N$:
We start with the full sequence.
In the $i$th step we keep the first $i$ elements of the subsequence from the previous step and reduce its remainder to a subsequence
along which $g_i*\mollifierTime_{\varepsilon_n}\to g_i$ pointwise on $[0,T]\setminus N_i$ for some nullset $N_i$.
This is possible since $g_i*\mollifierTime_{\varepsilon_n}\to g_i$ in $\mathrm{L}^1(\R)$.
Finally we set $N=\bigcup_{i\in\N}N_i$, which as a countable union of Lebesgue-nullsets is again a nullset.
By construction of the final subsequence, we have $g_i*\mollifierTime_{\varepsilon_n}\to g_i$ pointwise on $[0,T]\setminus N$ for all $i$,
thus $\int_{\dom}\varphi\dint(\density_{\varepsilon_n,t}-\densityT{t})\to0$ for all $\varphi\in\mathfrak C$ and almost all $t\in[0,T]$, as desired.
\end{proof}

\begin{remark}[H\"older versus absolute continuity]
Our construction in the previous proof would have simplified if we could have assumed $t\mapsto\density_{\varepsilon,t}$ to be absolutely instead of H\"older continuous
(i.e.\ $\mathbb W_2(\density_{\varepsilon,t},\density_{\varepsilon,s})\leq\int_t^sg(r)\dint r$ for some $\mathrm{L}^1$-function $g$).
Indeed, then an appropriate momentum would already have been provided by \cite[Thm.\,5.14]{OTAppliedMath}.
However, in general we cannot expect absolute continuity as the following example shows:
Consider $\density\in\measp([0,2]\times\dom)$ with $\densityT{t}=\delta_0$ for $t\leq1$ and $\densityT{t}=\delta_x$ else for some $0\neq x\in\dom$.
Then the support of $\density_{\varepsilon}$ is contained in $[0,2]\times\{0,x\}$,
but there cannot be any momentum $\momentum_\varepsilon$ absolutely continuous w.r.t.\ $\density_\varepsilon$ (in particular with same support)
satisfying the continuity equation \eqref{eqn:defCE}.
By \cite[Thm.\,5.14]{OTAppliedMath} this excludes absolute continuity of $t\mapsto\density_{\varepsilon, t}$.
\end{remark}

\begin{remark}[Application of \cref{thm:RecoverySequence}\ref{enm:RecoverySequenceB}]\label{rmk:SupConvergenceLipschitzFunctions}
A direct consequence of \cref{thm:RecoverySequence}\ref{enm:RecoverySequenceB} is that along the subsequence we have
\begin{equation*}
\sup_{x\in X}\absNorm{\int_\dom f(x,y)\dint\densityTN{t_n}{n}(y)-\int_\dom f(x,y)\dint\densityT{t}(y)}\convN 0
\end{equation*}
for almost every $t\in[0,T]$ and any Lipschitz function $f:X\times\dom\to\R$ with $X$ a metric space.
Indeed, denoting by $\mathrm{Lip}(h)$ the Lipschitz constant of a function $h$, by the Kantorovich--Rubinstein formula for $\mathbb W_1$ the supremum is bounded above by
$\sup_{\Phi:\dom\to\R,\,\mathrm{Lip}(\Phi)\leq\mathrm{Lip}(f)}\int_\dom\Phi\dint(\densityTN{t_n}{n}-\densityT{t})
=\mathrm{Lip}(f)\mathbb W_1(\densityTN{t_n}{n},\densityT{t})
\lesssim\mathbb W_2(\densityTN{t_n}{n},\densityT{t})$,
which converges to zero.
\end{remark}

\section{PET forward operator}\label{sec:ForwardOperator}
Here we briefly recapitulate the model for PET measurements from \cite{Schmitzer_DynamicCellImaging} and state some of its properties.
Recall that the radionuclide distribution $\density^\dagger=\d t\otimes\densityT{t}^\dagger$, the ground truth tracer distribution that we seek to approximately recover from the PET measurements, is located in some compact and convex domain $\dom\subset\R^3$. As in \cite{Schmitzer_DynamicCellImaging,PET_Base} we assume that the measurement times are much smaller than the radionuclide's halflife so that the tracer mass stays constant over the measurement time and hence we assume $\density^\dagger_t(D)$ to be independent of the time point $t$. The emitted photons are detected at the boundary of a convex set $\domDelta\supset\dom$ with $\mathrm{dist}(\dom,\partial\domDelta)\ge\delta$ and we will assume $\partial\domDelta$ to be smooth.
\begin{remark}[Generalization to piecewise smooth boundary]
  The results generalize to $\partial\domDelta$ having piecewise smooth boundary by generalizing \cref{thm:scatterlessDetection} below, assuming that the detectors are contained in exactly one smooth component of $\partial\domDelta$, and making obvious adjustments in the proofs relying on \cref{thm:scatterlessDetection}.
\end{remark}
A measurement $E_q$ consists of a list of photon pair detection times and locations.
It is a realization of the Poisson point process $\boldsymbol E_q=\Poi{q\d t\otimes \forwardOp\densityT{t}^\dagger}$.
The linear forward operator $A\colon\measp(\dom)\to\measp(\boundDomDeltaSq)$ describes transformation of radioactive decays into photon detections.
It is a weighted sum
\begin{equation*}
A=p^\att A^\att+p^\sct \forwardOp^\sct+p^\dt \forwardOp^\dt,
\end{equation*}
where the superscripts stand for
\textbf attenuation (the emitted photon pair is not detected, for instance due to absorption),
\textbf scattering (at least one of the photons is deflected), and
\textbf detection (the unscattered photons are registered by a pair of detectors).
The parameters $p^\sct,p^\att,p^\dt=1-p^\sct-p^\att\in[0,1]$ denote the corresponding probabilities.
%
The forward operator $A^\att$ simply discards all radioactive decays, 
\begin{equation*}
A^\att\colon\measp(\dom)\to\measp(\boundDomDeltaSq),\quad\densityT{t}^\dagger\mapsto0.
\end{equation*}
%
Random scattering is assumed in \cite{Schmitzer_DynamicCellImaging} to lead to a homogeneous coincidence probability on $\boundDomDeltaSq$,
\begin{equation*}
\forwardOp^\sct\colon \measp(\dom) \to \measp(\boundDomDeltaSq),\quad
\densityT{t}^\dagger \mapsto \frac{\densityT{t}^\dagger(\dom)}{\hd^2(\partial\domDelta)^2}
\cdot
(\HausdorffMeasSq)\restr(\boundDomDeltaSq),
\end{equation*}
but our analysis would also apply for alternative, spatially inhomogeneous scattering operators $\forwardOp^\sct$ (as long as $\forwardOp^\sct\density\gtrsim\norm{\density}$).
%
The forward operator of scatterless detection is the composition
\begin{equation*}
\forwardOp^\dt=B_{\mathrm{detectors}} B_{\mathrm{lines}} B_{\mathrm{pr}}.
\end{equation*}
Here, $B_{\mathrm{pr}}$ models the so-called \emph{positron range}:
The photon pair is emitted near, but not exactly at the location of the radioactive decay.
Therefore, the intensity $B_{\mathrm{pr}}\densityT{t}^\dagger$ of photon emissions
equals the convolution of the radioactive decay intensity $\densityT{t}^\dagger$ with the (Lipschitz) probability density $G:B_{\delta/2}(0)\to[0,\infty)$ of the location difference,
\begin{equation*}
B_{\mathrm{pr}}\colon\measp(\dom)\to\measp(\domDeltaHalf),\;
\densityT{t}^\dagger\mapsto G*\densityT{t}^\dagger
\qquad\text{for }\domDeltaHalf=\dom+B_{\delta/2}(0).
\end{equation*}
%
%
%
The operator $B_{\mathrm{lines}}$ maps the spatial intensity of photon pair emission
to the intensity of photon pair emission locations and directions
(a location-direction pair $(x,v)$ represents a photon pair emitted at $x\in\domDeltaHalf$ in directions $v\in S^2=\{x\in\R^3 \ | \ \absNorm{x}=1\}$ and $-v$),
\begin{equation*}
B_{\mathrm{lines}}\colon\measp(\domDeltaHalf)\to\measp(\domDeltaHalf\times S^2),\;
\mu\mapsto\mu\otimes\vol_{S^2}
\end{equation*}
with $\vol_{S^2}$ the uniform probability measure on the sphere.
Finally, each unscattered photon pair $(x,v)\in\domDeltaHalf\times S^2$ will be detected at the positions $R(x,v)$ with
\begin{equation}\label{eq:RFunctionForwardOperator}
R\colon\domDeltaHalf\times S^2\to\boundDomDeltaSq,\;
R(x,v)=\partial\domDelta\cap(x+\R v)
\end{equation}
(for simplicity, two-element subsets of $\partial\domDelta$ are identified with points in $\boundDomDeltaSq$).
This leads to
\begin{equation*}
B_{\mathrm{detectors}}\colon\measp(\domDeltaHalf\times S^2)\to\measp(\boundDomDeltaSq),\;
\mu\mapsto\pushforward{R}{\mu}.
\end{equation*}

For $\density=\d t\otimes\densityT{t}\in\measp(\spacetime)$ and $j\in\{\att,\sct,\dt\}$ or $j$ empty, we will denote by $\forwardOp^j\density$ the measure
\begin{equation*}
\forwardOp^j{\density}\in\measp([0,T]\times(\partial\domDelta)^2),\quad
\forwardOp^j{\density}(\tau\times\Gamma)=\int_\tau \forwardOp^j{\densityT{t}}(\Gamma)\dint t.
\end{equation*}
With a slight misuse of notation, we will use the same expression for the Radon--Nikodym derivative w.r.t.\ $\nu=\dt t\otimes(\hd^2\restr\partial\domDelta)\otimes(\hd^2\restr\partial\domDelta)$ (which will exist by \cref{thm:scatterlessDetection}\ref{enm:ForwardDensitySplitting}).

Given a ground truth radionuclide distribution $\density^\dagger\in\measp(\spacetime)$ and a PET measurement $E_q$ drawn from the random variable $\Poi{q\forwardOp\density^\dagger}$ (this means that on average we have $\normS{E_q}\simeq q$ detected events),
\cite{Schmitzer_DynamicCellImaging} proposes to reconstruct $\density^\dagger$ by minimizing
\begin{equation}\label{eqn:minimizationFunctionalDiscrete}
J^{E_q}(\density,\momentum) = \norm{A\density}-\frac{1}{q}\int\log\left(\forwardBinnedP{\density}\right)\d E_q+\beta\BBEnergy(\density,\momentum)
\end{equation}
where the radioactive intensity $q$ corresponds to the inverse halflife of the considered radionuclide and the discrete forward operator
\begin{equation}\label{eq:DefDiscreteForwardOperator}
\forwardBinnedP{\density}=\sum_{i=1}^N\sum_{j,k=1}^M\frac{\forwardP{\density}(\tau^i\times\Gamma^j\times\Gamma^k)}{\mathcal{L}^1(\tau^i)\hd^2(\Gamma^j)\hd^2(\Gamma^k)}\mathds{1}_{\tau^i\times\Gamma^j\times\Gamma^k}, \quad \forwardOpP=\UF p^\sct \forwardOp^\sct+p^\dt\forwardOpD.
\end{equation}
Above, $\mathds{1}_S$ denotes the characteristic function of a set $S$, $\Gamma^j\subset\partial\domDelta$, $j=1,\ldots,M$, denote the discrete photon detectors,
and $\tau^i\subset[0,T]$, $i=1,\ldots,N$, denote disjoint time intervals
(i.e.\ the operator models that the measurement can only be trusted up to the resolution of the detectors and the time intervals into which detected photon pairs are typically binned).
In \cite{Schmitzer_DynamicCellImaging,PET_Base} this reconstruction is derived as a maximum a posteriori estimate with some additional unbiasing procedure that results in the extra factor $\UF>0$. We also allow time intervals and detector pairs to depend on an index $n$ which will be indicated by $B_n^u$.

We next list some properties of the forward operator that were derived in \cite{PET_Base}.
One of these is the relation of the detection part $\forwardOp^\dt$ to the X-ray transform
\begin{multline*}\label{eq:XRayTrafo}
P:L^1(\domDeltaHalf)\to L^1(\mathcal C),
\quad
Pf(\theta,s)=\int_{\{r\in\R\,|\,s+r\theta\in\domDeltaHalf\}} f(s+r\theta)\dint\mathcal L(r),\\
\text{where }
\mathcal C=\{(\theta,s)\in S^2\times\R^3\,|\,s\in\pi_{\theta^\perp}(\domDeltaHalf)\}.
\end{multline*}
Here, $\theta^\perp=\{s\in\R^3\,|\,s\cdot\theta=0\}$ denotes the orthogonal complement of $\theta\in S^2$
and $\pi_{\theta^\perp}:\R^3\to\theta^\perp$ the orthogonal projection onto $\theta^\perp$.
Note that the X-ray transform satisfies the symmetry $Pf(\theta,s)=Pf(-\theta,s)$. 
On $\mathcal C$ we will use the Borel measure $\hd^2\otimes\mathcal L^2$, defined by dual pairing with any continuous function $f:\mathcal C\to\R$ as
\begin{equation*}
\int_{\mathcal C}f\dint(\hd^2\otimes\mathcal L^2)
=\int_{S^2}\int_{\theta^\perp}f(\theta,s)\dint\mathcal L^2(s)\dint\hd^2(\theta).
\end{equation*}
Also, for $(a,b)\in\boundDomDeltaSq$ we will abbreviate
\begin{equation*}
\theta(a,b)=\tfrac{b-a}{|b-a|}\in S^2,\qquad
s(a,b)=\pi_{\theta(a,b)^\perp}(a).
\end{equation*}
Finally, since the convolution $G\ast\lambda$ of some $\lambda\in\measp(\dom)$ with the continuous positron range kernel $G$ is absolutely continous with respect to $\mathcal L^3$,
we will identify it with its $\mathcal L^3$-density and write $P[G\ast\lambda]$.

\begin{lemma}[Scatterless detection, {\cite[Lemma 4.1-4.3]{PET_Base}}]\label{thm:scatterlessDetection}
For any $\lambda\in\measp(\dom)$ we have
\begin{enumerate}[label=(\alph*)]
\item\label{enm:ForwardDensitySplitting}
$\RNderivative{\forwardOp^\dt\lambda}{(\hd^2\restr\partial\domDelta)\otimes(\hd^2\restr\partial\domDelta)}(a,b)
=g(a,b)P[G\ast\lambda](\theta(a,b),s(a,b))$
for some bounded, smooth function $g:\boundDomDeltaSq\to(0,\infty)$,
\item\label{enm:BoundednessForwardOp}
$\RNderivative{\forwardOp^\dt\lambda}{(\hd^2\restr\partial\domDelta)\otimes(\hd^2\restr\partial\domDelta)}
\leq C\norm\lambda$ for some $C>0$ independent of $\lambda$.
\end{enumerate}
As a consequence of \ref{enm:BoundednessForwardOp}, for $\UF>0$ there exists $C>0$ such that for any $\density=\dt t\otimes\densityT{t}\in\measp(\spacetime)$ with $\densityT{t}(\dom)=\frac1T\norm{\density}$ for a.e.\ $t\in[0,T]$ we have
\begin{equation*}
\tfrac1C\norm\density
\leq\RNderivative{\forwardP{\density}}{\nu}
\leq C\norm\density
\qquad\text{and}\qquad
\tfrac1C\norm\density
\leq\RNderivative{\forwardBinnedP{\density}}{\nu}
\leq C\norm\density.
\end{equation*}
\end{lemma}


\begin{lemma}[Continuity of forward operator]\label{thm:ContinuityForwardOp}
Let $\density=\d t\otimes\densityT{t},\densityN{n}=\d t\otimes\densityTN{t}{n}\in\measp(\spacetime)$ such that $\densityN{n}\weakstarto\density$, then $\forwardOpP\densityN{n}\weakstarto \forwardOpP\density$.
For any interval $\tau\subset[0,T]$ and any measurable subsets $\Gamma^1,\Gamma^2\subset\partial\domDelta$ it also holds $\forwardOpP\densityN{n}(\tau\times\Gamma^1\times\Gamma^2)\to \forwardOpP\density(\tau\times\Gamma^1\times\Gamma^2)$.
The same statements hold with $\forwardOpP{}$ replaced by $\forwardBinnedP{}$.
\end{lemma}
\begin{proof}
We start with $\forwardOpP\densityN{n}\weakstarto \forwardOpP\density$. First, the results for the scatter part $A^\sct$ are a direct consequence of the weak-* convergence of $\densityN{n}$.
As for $\forwardOp^\dt$, for a continuous function $\varphi$ we find
\begin{align*}
\lim_{n\to\infty}\int_0^T\!\!\int_{\boundDomDeltaSq}\varphi \dint (\forwardOp^\dt\densityN{n})
=& \lim_{n\to\infty}\int_{S^2}\int_{\domDeltaHalf}\int_0^T\!\!\int_{\dom} \varphi(t, R(x,v))G(x-y)\dint\densityTN{t}{n}(y)\dint t\dint x\dint\volS(v)\\
=&\int_{S^2}\int_{\domDeltaHalf}\int_0^T\int_{\dom} \varphi(t, R(x,v))G(x-y)\dint\densityT{t}(y)\dint t\dint x\dint\volS(v)\\
=&\int_0^T\int_{\boundDomDeltaSq}\varphi \dint (\forwardOp^\dt\density)
\end{align*}
by compactness of the spaces, continuity of $R$ and uniform boundedness of $\norm{\densityN{n}}$. The second assertion follows in a similar way by considering
\begin{align*}
	\forwardOpD\densityN{n}(\tau\times\Gamma^1\times\Gamma^2)=\int_{S^2}\int_{\domDeltaHalf}\mathds{1}_{R^{-1}(\Gamma^1\times\Gamma^2)}(x,v)\int_\tau\int_{\dom}G(x-y)\dint\densityTN{t}{n}(y)\dint t\dint x\dint\volS(v)
\end{align*}
and convergence follows from the Portmanteau Theorem \cite[Thm.\,13.16]{KlenkeProbabilityTheory} applied to $(t,y)\mapsto \mathds{1}_\tau(t)G(x-y)$ because the sets of discontinuities of these functions have measure zero with respect to $\density$. Finally, the transfer to $\forwardBinnedP{}$ is immediate.
\end{proof}


The final result in this section will be needed to study the effect of increasing detector resolution. 

\begin{lemma}[Finer discretization]\label{thm:ConvDisretization}
Let $X$ be a compact metric space with its Borel $\sigma$-algebra and consider the measure $\mu\in\measp(X)$.
Let $X=\bigcup_{k=1}^{N_n}C_n^k$, $n\in\N$, be a sequence of finite partitions with $\max_k\mathrm{diam}(C_n^k)\convN0$.
For a sequence $(g_n)_n\subset \mathrm{L}^2(X)$ of measurable functions on $X$ define
\begin{align*}
G_n\in\mathrm{L}^\infty(X),\quad
x\mapsto\sum_k\frac{\int_{C_n^k}g_n\dint\mu}{{\mu(C_n^k)}}\mathds{1}_{C_n^k}(x),
\end{align*}
where $G_n$ is defined to be zero on sets $C_n^k$ with $\mu(C_n^k)=0$. 
If $g_n\to g$ weakly in $\mathrm{L}^2(X)$, then $G_n\to g$ weakly in $\mathrm{L}^2(X)$.
\end{lemma}
\notinclude{
\begin{remark}[Vitali relation]
	The additional assumption $\absNorm{\frac{1}{\mu(C_n^{k_n})}\int_{C_n^{k_n}}g\dint\mu-g(x)}\to0$
	is for instance satisfied if $(C_n^k)_{n,k}$ is a $\mu$-Vitali relation (cf.\ \cref{AppendixVitaliRelation}).
\end{remark}
}

\begin{proof}
First note that $G_n$ is well-defined, because every summand is finite by Jensen's inequality,
$[\int_{C_n^k}g_n\dint\mu/\mu(C_n^k)]^2\leq\int_{C_n^k}g_n^2\dint\mu/\mu(C_n^k)\leq\norm{g_n}_{\mathrm{L}^2}^2/\mu(C_n^k)$.

We show $\lim_{n\to\infty}\int_X\varphi G_n\dint\mu=\int_X\varphi g\dint\mu$ for all $\varphi\in C(X)$.
By density of $C(X)$ in $L^2(X)$ \cite[Prop.\,7.9]{Folland_RealAnalysis} this implies the desired weak convergence.
\notinclude{First, 
by Hölder's inequality it is
\begin{align*}
\norm{G_n}_{\mathrm{L}^2}^p&=\int_X\left( \sum_k\frac{1}{{\mu(C_n^k)}}\int_{C_n^k}g_n(y)\dint\mu(y)\mathds{1}_{C_n^k}(x) \right)^2\dint\mu(x)\\
&=\sum_k\frac{1}{\mu(C_n^k)}\left( \int_{C_n^k}g_n(y)\dint\mu(y) \right)^2\\
&\le\sum_k\left( \int_{C_n^k}g_n^2(y)\dint\mu(y) \right)=\int_{X}g_n^2(y)\dint\mu(y)\lesssim1.
\end{align*}
}
Let $p_n^k\in C_n^k$\notinclude{ for all $n,k$}, then
\begin{multline*}
	\int_X(G_n-g)\varphi\dint\mu
	=\int_X(g_n(y)-g(y))\varphi(y)\dint\mu(y)\\
	+\underbrace{\sum_k\int_{C_n^k}\frac{\int_{C_n^k}g_n\dint\mu}{\mu(C_n^k)}(\varphi(x)-\varphi(p_n^k))\dint\mu(x)}_{=\colon\text{error}_1}
	+\underbrace{\sum_k\int_{C_n^k}g_n(y)(\varphi(p_n^k)-\varphi(y))\dint\mu(y)}_{=\colon\text{error}_2}
\end{multline*}
%
Since $X$ is compact, $\varphi$ is uniformly continuous. Thus, for every $\varepsilon>0$ there exists a $\delta>0$ such that
\begin{align*}
\absNorm{\varphi(x)-\varphi(\tilde{x})}<\varepsilon \quad\text{whenever}\quad \dist(x,\tilde{x})<\delta.
\end{align*}
Now let $N(\delta)$ large enough so that $\mathrm{diam}(C_n^k)<\delta$ for all $n\ge N(\delta)$, then for those $n$
\begin{equation*}
\absNorm{\text{error}_1}+\absNorm{\text{error}_2}
\le\varepsilon\sum_k\left(\int_{C_n^k}\absNorm{\frac{\int_{C_n^k}g_n\dint\mu}{\mu(C_n^k)}}\dint\mu
+\int_{C_n^k}\absNorm{g_n}\dint\mu\right)
\leq2\varepsilon\int_X\absNorm{g_n}\dint\mu
\lesssim\varepsilon
\end{equation*}
so that $\int G_n\varphi\dint\mu\to\int g\varphi\dint\mu$ and thus $G_n\xrightharpoonup{\mathrm{L}^2} g$.
%
\end{proof}

\section{$\Gamma$-convergence}\label{SectionGammaConvergence}
In this section we analyze the stability of the PET reconstruction model \eqref{eqn:minimizationFunctionalDiscrete} by means of $\Gamma$-convergence as the SNR tends to infinity.
This can mathematically be expressed by an increasing measurement intensity $\intensity\to\infty$ (corresponding to a decreasing halflife $\halflife=\frac{1}{q}\to 0$).
We take the minimum over all possible momentum measures $\momentum$ in \eqref{eqn:minimizationFunctionalDiscrete} as we are mostly interested in the stability of the reconstructed density $\density$.
(Note that if we also send the regularization parameter $\beta$ to zero, as we typically would in the vanishing noise limit, then stability of $\momentum$ cannot be expected since it is no longer regularized and no compactness results can be obtained.)
For a given measurement $E_{q_n}$ (i.e.\ a sum of Dirac measures on $[0,T]\times\boundDomDeltaSq$) this leads to the functional
\begin{align*}
\mathcal{E}^{E_{q_n}}_n(\density)&=\norm{A\density} - \frac{1}{\intensity_n}\int\Log{\forwardBinnedPN{\density}{n}}\d E_{q_n}+\beta_n\min_\momentum\BBEnergy(\density,\momentum)\\
&=\norm{\forward{\density}}-\frac{1}{q_n}\sum_{i=1}^{N_n}\sum_{j,k=1}^{M_n}\Log{\frac{\forwardPN{\density}{n}(\tau_n^i\times\Gamma_n^j\times\Gamma_n^k)}{\mathcal{L}^1(\tau_n^i)\hd^2(\Gamma_n^j)\hd^2(\Gamma_n^k)}} E_{q_n}(\tau_n^i\times\Gamma_n^{j}\times\Gamma_n^k) +\beta_n\min_\momentum\BBEnergy(\density,\momentum),
\end{align*}
where all quantities that may vary with the intensities $\intensity_n\to\infty$ are now indexed by $n$ 
(we drop the index if we do not refer to a sequence).
Note that minimizing $\mathcal{E}^{E_{q_n}}_n$ is equivalent to minimizing \eqref{eqn:minimizationFunctionalDiscrete}.


\begin{theorem}[Existence of minimizers, {\cite[Thm.\,4.7]{PET_Base}}]\label{TheoremMinExistence}
	Let $\beta,p^\dt,p^\sct,q, u>0$ and let the measurement $E_q$ be a realization of $\bm E_q=\Poi{q\forwardOp\density^\dagger}$
	for some ground truth material distribution $\density^\dagger=\dt t\otimes\densityT{t}^\dagger\in\measp(\spacetime)$.
	Then almost surely (in particular if $\norm{E_q}<\infty$) the set of minimizers of $J^{E_q}$ and $\mathcal{E}^{E_q}_n$ is non-empty and compact with respect to weak-* convergence.
\end{theorem}

To prove convergence of our reconstruction to the ground truth in the vanishing noise limit,
we will need to infinitely refine the detectors and time bins at the same time.
However, from an application viewpoint it might also be of interest to keep the spatial and/or temporal resolution finite.
To more easily treat all these cases together (despite e.g.\ the limit forward operators being different),
we give labels to the different situations:
\begin{enumerate}[label=(\Alph*)]
	\item\label{enm:TauConstGammaConst} $\tau_n^i=\tau^i$ and $\Gamma_n^{k}=\Gamma^{k}$ constant in $n$
	\item\label{enm:TauNonconstGammaConst} $\max_i\mathcal{L}^1(\tau_n^i)\rightarrow 0$ and $\Gamma_n^{k}=\Gamma^{k}$ constant in $n$
	\item\label{enm:TauConstGammaNonconst} $\max_k\mathrm{diam}(\Gamma^k_n)\rightarrow 0$ and $\tau_n^i=\tau^i$ constant in $n$
	\item\label{enm:TauNonconstGammaNonconst} $\max_i\mathcal{L}^1(\tau_n^i)\rightarrow 0$ and $\max_k\mathrm{diam}(\Gamma^k_n)\rightarrow 0$
\end{enumerate}
The corresponding limit forward operators (which provide the intensity of photon pair detection on $[0,T]\times\partial\domDelta\times\partial\domDelta$ as a density w.r.t.\ $\nu=\dt t\otimes\HausdorffMeasSq$) are expressed via the same symbol as
\begin{align*}
\forwardBinnedLimit{\density}(t,x,y)=\begin{cases}
\forwardBinnedP{\density}(t,x,y)\ \  &\text{\ref{enm:TauConstGammaConst},}\\
\sum_{j,k=1}^M\int_{\Gamma^j\times\Gamma^k}\RNderivative{\forwardP{\density}}{\nu}(t, a, b)\frac{\d \HausdorffMeas{a, b}}{\hd^2(\Gamma^j)\hd^2(\Gamma^k)}\mathds{1}_{\Gamma^j\times\Gamma^k}(x,y) \ \  &\text{\ref{enm:TauNonconstGammaConst},}\\
\sum_{i=1}^N\int_{\tau^i}\RNderivative{\forwardP{\density}}{\nu}(s, x, y)\frac{\dint s}{\mathcal{L}^1(\tau^i)} \mathds{1}_{\tau^i}(t) \ \  &\text{\ref{enm:TauConstGammaNonconst},}\\
\RNderivative{\forwardP{\density}}{\nu}(t,x,y)\ \  &\text{\ref{enm:TauNonconstGammaNonconst}.}
\end{cases}
\end{align*}

If one considers a limit in which the regularization stays active
(which may for instance be advisable to improve reconstruction quality if the detector or time interval sizes stay bounded away from zero),
then the corresponding $\Gamma$-convergence result is fairly straightforward.

\begin{theorem}[$\Gamma$-convergence for active limit regularization]\label{thm:GammaConv2}
Let $q_n\to\infty$, $\UF_n\to\UF>0$, and $\beta_n\to\beta>0$ monotonically as $n\to\infty$ for real positive sequences $q_n,\UF_n,\beta_n$.
Further, given $0\neq\density^\dagger\in\meas_c$  (cf.\ \eqref{eqn:constantMass}), let $\bm E_{q_n}$ be a PPP with intensity measure $q_n\forward{\density^\dagger}>0$ such that
\begin{enumerate}[label=(\alph*)]
\item
$\bm E_{q_n}$ is an arbitrary sequence of PPP and $\log(n)\in o(q_n)$, or
\item
$\bm E_{q_n}$ is obtained from stochastic coupling.
\end{enumerate}
Then almost surely, the $\Gamma$-limit of $\mathcal{E}^{\bm E_{q_n}}_n$ w.r.t.\ weak-* convergence is
\begin{equation*}
\Gamma-\lim_{n\to\infty}\mathcal{E}^{\bm E_{q_n}}_n=\mathcal{E}_\infty
\qquad\text{for}\quad
\mathcal{E}_\infty(\density)=\norm{\forward{\density}}-\int\Log{\forwardBinnedLimit{\density}}\dint\forward{\density^\dagger}+\beta\min_\momentum\BBEnergy(\density,\momentum)+\iota_{\meas_c}(\density),
\end{equation*}
where $\iota_{\meas_c}$ is the convex indicator function of $\meas_c$. 
\end{theorem}
\begin{proof}
We first show that if $\densityN{n}\weakstarto\density$ with uniformly bounded $\BBEnergy(\densityN{n},\momentum_n)\leq C$,
then $\forwardBinnedPN{\densityN{n}}{n}\to\forwardBinnedLimit{\density}$ uniformly.
Indeed, by \cite[Lemma 4.4]{PET_Base} (this result requires $\momentum_n\weakstarto\momentum$ only to get the disintegration $\density=\d t\otimes\densityT{t}$, but the latter already follows from closedness of $\meas_c$ by \cref{thm:closednessOfBSet})
we have $\forwardP\density_n\in C(\spacetimedet)$ and $\forwardP\density_n\to \forwardP\density$ uniformly by the uniform bound on $\BBEnergy$. Due to $\norm{\densityN{n}}\to\norm{\density}$ (\cref{thm:closednessOfBSet}) this also implies $\forwardBinnedPN{(\densityN{n}-\density)}{n}\to0$ uniformly.
Since $\forwardBinnedPN{\density}{n}\to\forwardBinnedLimit{\density}$ uniformly
(recall that the uniform convergence $\forwardP\density_n\to\forwardP\density$ implies continuity of $\forwardP\density$),
we obtain the desired uniform convergence $\forwardBinnedPN{\densityN{n}}{n}\to\forwardBinnedLimit{\density}$.

Now consider the liminf inequality.
Let $\densityN{n}\weakstarto\density$ with $\liminf_n\mathcal{E}^{\bm E_{q_n}}_n(\densityN{n})<\infty$ (otherwise there is nothing to show).
We may even assume the liminf to be a limit (else we pass to a subsequence).
We have $\liminf_n\norm{\forwardOp\densityN{n}}\geq\norm{\forwardOp\density}$ due to weak-* lower semi-continuity of the norm.
Likewise $\liminf_n\beta_n\min_\momentum\BBEnergy(\densityN{n},\momentum)\geq\beta\min_\momentum\BBEnergy(\density,\momentum)$ by \cref{thm:continuityBBEnergy}.
Finally, due to the uniform convergence $\forwardBinnedPN{\densityN{n}}{n}\to\forwardBinnedLimit{\density}$ from above
and almost sure weak-* convergence $\bm E_{q_n}/q_n\weakstarto\forward{\density^\dagger}$ by \cref{thm:flatConvergence}
we have $\int\Log{\forwardBinnedPN{\density}{n}}\dint \bm E_{q_n}/q_n\to\int\Log{\forwardBinnedLimit{\density}}\dint\forward{\density^\dagger}$, as desired.

For the limsup inequality take the constant sequence $\densityN{n}=\density$,
then $\lim_n\mathcal{E}^{\bm E_{q_n}}_n(\densityN{n})=\mathcal{E}_\infty(\density)$ follows again
from the uniform convergence $\forwardBinnedPN{\densityN{n}}{n}\to\forwardBinnedLimit{\density}$
(as long as $\min_\momentum\BBEnergy(\density,\momentum)<\infty$ -- otherwise there is nothing to show).
%
\end{proof}

However, if SNR and spatiotemporal resolution tend to infinity, the regularization should vanish in the limit not to artificially distort the reconstruction.
This setting is more complicated and requires some preparation.
During the remainder of the article we will impose the following uniformity and compatibility conditions on the detectors and time intervals.
\begin{assumption}[Uniform and compatible refinement]\label{AssumptionSetting}
For each level $n$ we assume:
\begin{enumerate}[label=(\roman*)]
\item\label{enm:timePartition}
The time intervals $(\tau_n^i)_{i=1,\ldots,N_n}$ form a partition of $[0,T]$.
The partition is quasiuniform in the sense that there exist $c,c'>0$ independent of $n$ with $c'\leq N_n\mathcal{L}^1(\tau_n^i)\leq c$ for all $i$.
Moreover, the partition is a refinement of the previous partition,
i.e.\ each $\tau_n^i$ lies within one and only one $\tau_{n-1}^j$.
\item
The detectors $(\Gamma_n^k)_{k=1,\ldots,M_n}$ are Borel and path connected and
they form a partition of $\partial\domDelta$.
The partition is quasiuniform in the sense that there exist $c,c'>0$ independent of $n$ with $c'\leq M_n\mathrm{diam}(\Gamma_n^k)^2\leq c$ for all $k$.
Moreover, the partition is a refinement of the previous partition,
i.e.\ each $\Gamma_n^k$ lies within one and only one $\Gamma_{n-1}^l$.
\end{enumerate}
\end{assumption}
Obviously, the quasiuniformity assumptions are only relevant if the time interval or detector sizes tend to zero.
\notinclude{
In the latter setting, the assumption on the detectors guarantees that 
\begin{equation}\label{VitaliRelationOfDetectors}
V=\left\{\left((x,y), \Gamma'\times \Gamma''\right) \ \middle| \ (x,y)\in\Gamma'\times \Gamma'', \ \Gamma',\Gamma''\in\left({\Gamma_n^k}\right)_{n,k}\right\}
\end{equation}
is a $\hd^2\restr\partial\domDelta\otimes \hd^2\restr\partial\domDelta$ Vitali relation by \cref{thm:VitaliRelationFederer}.
}

For the $\Gamma$-convergence \cref{thm:GammaConv} with vanishing regularization we need some preparatory results. First, we note that mass preservation in time (which is implied by our regularization $\BBEnergy$ and which will be the only regularization in the limit functional for vanishing limit regularization, see \cref{thm:GammaConv}) is conserved in the weak-* limit by \cref{thm:closednessOfBSet}.
Furthermore, the next result will be used to prove the liminf inequality in the $\Gamma$-convergence.
\begin{lemma}[Convergence of forward operator]\label{thm:Mazur}
Let $\densityN{n}\convStar\density\neq0$ in $\meas_c\subset\measp(\spacetime)$ and $\UF_n\rightarrow \UF\ge 0$.
Then for all cases \ref{enm:TauConstGammaConst}-\ref{enm:TauNonconstGammaNonconst}
\begin{gather*}
\forwardBinnedPN{\densityN{n}}{n}\to\forwardBinnedLimit{\density}
\text{ weakly in }\mathrm{L}^2([0,T]\times\boundDomDeltaSq)
\qquad\text{and}\\
\int_0^T\int_{(\partial\domDelta)^2}\Log{\forwardBinnedLimit{\density}}\RNderivative{\forward{\density^\dagger}}{\nu}\dint\HausdorffMeasSq\dint t
\ge\limsup_{n\rightarrow\infty}\int_0^T\int_{(\partial\domDelta)^2}\Log{\forwardBinnedPN{\densityN{n}}{n}}\RNderivative{\forward{\density^\dagger}}{\nu}\dint\HausdorffMeasSq\dint t.
\end{gather*}
\end{lemma}
\begin{proof}
We only prove case \ref{enm:TauNonconstGammaNonconst}, the other cases being simple modifications.
By \cref{thm:closednessOfBSet} we have $\density=\d t\otimes\densityT{t}$ with $\densityT{t}(\dom)$ independent of $t$ (up to Lebesgue-nullsets).
The weak-*-weak-* continuity of $\forwardOp^\sct$ and $\forwardOp^\dt$ by \cref{thm:ContinuityForwardOp} implies $\forwardPN{\densityN{n}}{n}\convStar\forwardP{\density}$,
and by density of continuous functions in $\mathrm{L}^2$ we also obtain $\d{\forwardPN{\densityN{n}}{n}}/\d{\nu}\to\d{\forwardP{\density}}/\d{\nu}$ weakly in $\mathrm{L}^2$.
An application of \cref{thm:ConvDisretization} thus proves $\forwardBinnedPN{\densityN{n}}{n}{\to} \forwardBinnedLimit{\density}$ weakly in $\mathrm{L}^2$.

The inequality finally follows from the sequential weak lower semi-continuity of convex integral functionals \cite[Example 1.23]{DalMasoGammaConv},
in particular of the functional
\begin{equation*}
\mathrm{L}^2\ni u\mapsto-\int_0^T\int_{(\partial\domDelta)^2}\Log{u}\RNderivative{\forward{\density^\dagger}}{\nu}\dint\HausdorffMeasSq\dint t.
\qedhere
\end{equation*}
\end{proof}

In case the regularization strength vanishes in the limit (see \cref{thm:GammaConv}) we cannot fully make use of the time regularity implied by the regularizer $\BBEnergy$ since $\min_\momentum\BBEnergy(\densityN{n}, \momentum)$ may diverge. Nevertheless, the regularizer still induces Hölder continuity in time along the sequence which can be exploited to improve the results for cases \ref{enm:TauNonconstGammaConst} and \ref{enm:TauNonconstGammaNonconst} (see also \cref{rmk:ImprovedConvergenceConditions}).
Indeed, the following result is provided in the proof of \cite[Lemma\,4.4]{PET_Base}.
\begin{lemma}[Hölder continuity in time]\label{lemma:TimeLipschitzForwardOperator}
	Let $\density=\d t\otimes\densityT{t}\in\measp(\spacetime)$ and $\momentum\in\meas([0,T]\times\R^3)^3$ such that $\BBEnergy(\density,\momentum)<\infty$. Then for $s,t\in[0,T]$ and $\Gamma,\Gamma'\subset\partial\domDelta$ measurable and path connected,
	\begin{align*}
		\tfrac{1}{\hd^2(\Gamma)\hd^2(\Gamma')}\absNorm{\forwardOpD\densityT{t}(\Gamma\times\Gamma')-\forwardOpD\density_s(\Gamma\times\Gamma')}\lesssim\absNorm{t-s}^{\frac12}\norm{\density}^{\frac12}\BBEnergy(\density,\momentum)^{\frac12}.
	\end{align*}
\end{lemma}

\begin{theorem}[$\Gamma$-convergence for vanishing regularization]\label{thm:GammaConv}
The $\Gamma$-convergence of \cref{thm:GammaConv2} even holds for $\beta=0$
if in cases \ref{enm:TauNonconstGammaConst} and \ref{enm:TauNonconstGammaNonconst} we additionally assume $1/\beta_{n+1}\in o(q_n)$ or $N_n\in o(q_n)$.
\end{theorem}

\begin{remark}[Regularization strength]
The Benamou--Brenier regularization in $\mathcal{E}^{\bm E_{q_n}}_n$ serves to temporally connect the different measured coincidences:
From a single coincidence one cannot localize a mass particle (one only knows the line segment on which it approximately lies),
but if the mass particle did not move too fast, as encoded in the Benamou--Brenier regularization,
then two or more temporally close coincidences from the same particle allow a good estimate of where it is.
However, if the regularization weight decreases faster than the number of coincidences increases,
then the coincidences (which all happen at different time points) no longer allow any localization of a mass particle:
Between two coincidences the particle might have moved arbitrarily far.
Therefore, the condition $1/\beta_{n+1}\in o(q_n)$ is expected and natural.
If the Benamou--Brenier regularization is too weak (it cannot be omitted, though), one can alternatively achieve the temporal connection of coincidences via binning into time intervals $\tau^i$:
Here, too, the time intervals may not shrink faster than the number of coincidences increases, resulting in the alternative condition $N_n\in o(q_n)$
(which in cases \ref{enm:TauConstGammaConst} and \ref{enm:TauConstGammaNonconst} is trivially satisfied).
\end{remark}

\begin{proof}
Since $\norm{\density^\dagger}<\infty$ we have $\norm{\bm E_{q_n}}<\infty$ almost surely. Hence it suffices to restrict to finite measurements in the following.
\paragraph{liminf inequality.}
Let $\densityN{n}\convStar\density$.
First consider the case $\density=0$, thus\ $\mathcal{E}_\infty(\density)=\infty$.
For a contradiction, assume $\liminf_{n\to\infty}\mathcal{E}^{\bm E_{q_n}}_n(\densityN{n})<\infty$.
By passing to a subsequence (still indexed by $n$) we may also assume $\mathcal{E}^{\bm E_{q_n}}_n(\densityN{n})<\infty$ for all $n$.
Due to $\mathcal{E}^{\bm E_{q_n}}_n(\densityN{n})\geq\beta_n\min_\momentum\BBEnergy(\densityN{n},\momentum)\geq\iota_{\mathcal M_c}(\densityN{n})$ by \cref{rem:propertiesBenamouBrenier}\ref{enm:massConservation}
we obtain $\densityN{n}\in\mathcal M_c$ for all $n$.
Now \cref{thm:scatterlessDetection} yields
\begin{equation*}
\mathcal{E}^{\bm E_{q_n}}_n(\densityN{n})
\geq-\frac{1}{q_n}\int\Log{\forwardBinnedPN{\densityN{n}}{n}}\dint \bm E_{q_n}
\geq-\frac{1}{q_n}\int\Log{C\norm{\densityN{n}}}\dint \bm E_{q_n}
=-\Log{C\norm{\densityN{n}}}\frac{\norm{\bm E_{q_n}}}{q_n}.
\end{equation*}
Since $\norm{\densityN{n}}\to0$ as $n\to\infty$ as well as $\norm{\bm E_{q_n}}/q_n\to\norm{\forward{\density^\dagger}}>0$ almost surely by \cref{thm:stochConv},
we obtain the contradiction $\mathcal{E}^{\bm E_{q_n}}_n(\densityN{n})\to\infty$, so the liminf inequality holds.

Now let $\densityN{n}\convStar\density\neq0$.
Assume first $\mathcal{E}_\infty(\density)=\infty$, then necessarily $\density\notin\meas_c$, because $\forwardBinnedLimit{\density}$ is bounded away from zero for $\density\neq0$ by \cref{thm:scatterlessDetection} and thus the integral in $\mathcal{E}_\infty(\density)$ is finite.
Due to the closedness of $\meas_c$ by \cref{thm:closednessOfBSet} this means $\densityN{n}\in\meas_c$ for finitely many $n$ only (otherwise there would exist a subsequence $\densityN{n_j}\in\meas_c$, $\densityN{n_j}\convStar\density$, implying $\density\in\meas_c$).
By \cref{rem:propertiesBenamouBrenier}\ref{enm:massConservation} this implies $\liminf_n\mathcal{E}^{\bm E_{q_n}}_n(\densityN{n})\geq\liminf_n\min_\momentum\BBEnergy(\densityN{n},\momentum)\ge\liminf_n\iota_{\meas_c}(\density_n)=\infty$.

Now consider $\mathcal{E}_\infty(\density)<\infty$. Without loss of generality we may assume $\liminf_n\mathcal{E}^{\bm E_{q_n}}_n(\densityN{n})<\infty$ (else there is nothing to show) and even $\mathcal{E}^{\bm E_{q_n}}_n(\densityN{n})<\infty$ for every $n\in\mathbb{N}$ (else pass to a subsequence). Consider the three summands of $\mathcal{E}^{\bm E_{q_n}}_n(\densityN{n})$. From \cref{thm:ContinuityForwardOp}, the weak-* continuity of the forward operator, we get the convergence of the first summand,
\begin{equation*}
\norm{\forward{\densityN{n}}}\rightarrow\norm{\forward{\density}}.
\end{equation*}
For the last summand, \cref{rem:propertiesBenamouBrenier}\ref{enm:massConservation} implies
\begin{equation*}
\beta_n\min_\momentum\BBEnergy(\densityN{n},\momentum)
\ge\iota_{\meas_c}(\densityN{n})
\end{equation*}
with the right-hand side being lower semi-continuous as $n\to\infty$ by the weak-* closedness of $\meas_c$ due to \cref{thm:closednessOfBSet}. It remains to deal with the middle summand. We expand
\begin{multline}\label{eq:GammaEffectiveQuantities}
-\frac{1}{q_n}\int\Log{\forwardBinnedPN{\densityN{n}}{n}}\dint \bm E_{q_n}=-\int\Log{\forwardBinnedPN{\densityN{n}}{n}}\RNderivative{\forward{\density^\dagger}}{\nu}\dint\HausdorffMeasSq\dint t\\
+\left(\int\Log{\forwardBinnedPN{\densityN{n}}{n}}\RNderivative{\forward{\density^\dagger}}{\nu}\dint\HausdorffMeasSq\dint t -\frac{1}{q_n}\int\Log{\forwardBinnedPN{\densityN{n}}{n}}\dint \bm E_{q_n} \right).
\end{multline}
By \cref{thm:Mazur}, the first term satisfies
\begin{align*}
\liminf_{n\to\infty}-\int\Log{\forwardBinnedPN{\densityN{n}}{n}}\RNderivative{\forward{\density^\dagger}}{\nu}\dint\HausdorffMeasSq\dint t\ge-\int\Log{\forwardBinnedLimit{\density}}\RNderivative{\forward{\density^\dagger}}{\nu}\dint\HausdorffMeasSq\dint t,
\end{align*}
so it remains to show that the term in parentheses vanishes in the limit.
To this end we introduce larger artificial detectors and time intervals. For $n$ large and $\varepsilon\in(0,1)$ we consider the index shift functions $I,J,K\colon\N\to\N\cup\{\infty\}$,
\begin{align*}
K(n)&=\min\left\{ k\in\N \ \middle| \  q_{n-1}/N_k\le(\beta_nq_{n-1})^{\varepsilon} \right\},\\
I(n)&=n\text{ in case \ref{enm:TauConstGammaConst} or \ref{enm:TauConstGammaNonconst} or if }N_n\in o(q_n)
\quad\text{ and }I(n)=\min\{n,K(n)\}\text{ else,}\\
J(n)&=n\text{ in case \ref{enm:TauConstGammaConst} or \ref{enm:TauNonconstGammaConst}}
\quad\text{ and }J(n)=\max\left\{ k\in\N \ \middle| \ k\le n, \ M_k\le(q_n/N_{I(n)})^{1/4} \right\}\text{ else,}
\end{align*}
which satisfy $I(n),J(n)\le n$.
Note that $K(n)$ is finite if $N_n$ diverges.
Furthermore, we always have $q_n/N_{I(n)}\to\infty$ (consequently $J(n)$ is well-defined for $n$ large enough):
Indeed, this is obvious in case \ref{enm:TauConstGammaConst} or \ref{enm:TauConstGammaNonconst} or if $N_n\in o(q_n)$ so that the only remaining case to check is when $1/\beta_{n+1}\in o(q_n)$.
Due to $N_{K(n)}\geq N_{I(n)}$ it suffices to show $q_n/N_{K(n)}\to\infty$.
For a contradiction, assume there is a subsequence, still indexed by $n$, with $q_n/N_{K(n)}$ bounded.
Note that $K(n)$ diverges by definition, hence we can extract yet another subsequence, again still indexed by $n$, such that $K(n+1)-1\geq K(n)$ for all $n$.
Therefore, $q_n/N_{K(n)}\geq q_n/N_{K(n+1)-1}>(\beta_{n+1}q_{n})^{\varepsilon}$ by definition of $K(n+1)$, however, the right-hand side diverges, proving $q_n/N_{I(n)}\to\infty$.
It is also readily checked that for $n$ large enough we always have $M_{J(n)}\le(q_n/N_{I(n)})^{1/4}$.
We next define the sets
\begin{align*}
	S_{n, J(n)}^k\coloneqq\left\{ \Gamma_n^l  \ \middle| \ l\in\{1,\ldots,M_n\},\ \Gamma_n^l\subset\Gamma_{J(n)}^k \right\},
	\quad T_{n, I(n)}^k\coloneqq\left\{ \tau_n^l \ \middle| \ l\in\{1,\ldots,N_n\},\ \tau_n^l\subset\tau_{I(n)}^k \right\}
\end{align*}
that consist of all detectors $\Gamma$ (time intervals $\tau$) at level $n$ that are subsets of a detector $\Gamma^k_{J(n)}$ (time interval $\tau^k_{I(n)}$) at level $J(n)$ (or $I(n)$). We then write
\begin{align}
	&\int\Log{\forwardBinnedPN{\densityN{n}}{n}}\left(\RNderivative{\forward{\density^\dagger}}{\nu}\dint\HausdorffMeasSq\dint t - \dint \frac{\bm E_{q_n}}{q_n}\right)\nonumber\\
	=&\sum_{i=1}^{N_n}\sum_{j,k=1}^{M_n}\Log{\frac{\forwardOpPN{\densityN{n}}(\tau_n^i\times\Gamma_n^j\times\Gamma_n^k)}{\mathcal{L}^1(\tau_n^i)\hd^2(\Gamma_n^j)\hd^2(\Gamma_n^k)}}\left[ \forwardOp\density^\dagger(\tau_n^i\times\Gamma_n^j\times\Gamma_n^k) -\frac{\bm E_{q_n}(\tau_n^i\times\Gamma_n^j\times\Gamma_n^k)}{q_n}\right]\label{eq:Gamma0}\\
	=&\sum_{i=1}^{N_{\!I(n)}}\sum_{j,k=1}^{M_{\!J(n)}}\sum_{\tau\in T^i_{n,I(n)}\!}\sum_{\Gamma\in S^j_{n,J(n)}\!}\sum_{\Gamma'\in S^k_{n,J(n)}\!}
	\!\!\left[\! \log\!\left(\!\frac{\forwardOpPN{\densityN{n}}(\tau\times\Gamma\times\Gamma')}{\mathcal{L}^1(\tau)\hd^2(\Gamma)\hd^2(\Gamma')}\!\right)
	\!-\! \log\!\left(\!\frac{\forwardOpPN{\densityN{n}}(\tau_{\!\!I(n)\!}^i\times\Gamma_{\!\!J(n)\!}^j\times\Gamma_{\!\!J(n)\!}^k)}{\mathcal{L}^1(\tau^i_{\!\!I(n)\!})\hd^2(\Gamma_{\!\!J(n)\!}^j)\hd^2(\Gamma_{\!\!J(n)\!}^k)}\!\right) \!\right]\nonumber\\
	&\qquad\qquad\qquad\qquad\qquad\qquad\qquad\qquad\qquad\cdot\left[\forwardOp\density^\dagger(\tau\times\Gamma\times\Gamma')- \tfrac1{q_n}{\bm E_{q_n}(\tau\times\Gamma\times\Gamma')} \right]\nonumber\\
	+&\sum_{i=1}^{N_{I(n)}}\sum_{j,k=1}^{M_{J(n)}} \Log{\frac{\forwardOpPN{\densityN{n}}(\tau_{I(n)}^i\times\Gamma_{J(n)}^j\times\Gamma_{J(n)}^k)}{\mathcal{L}^1( \tau^i_{I(n)})\hd^2(\Gamma_{J(n)}^j)\hd^2(\Gamma_{J(n)}^k)}}\nonumber\\
	&\qquad\qquad\qquad\qquad\qquad\cdot\left[\forwardOp\density^\dagger(\tau_{I(n)}^i\times\Gamma_{J(n)}^j\times\Gamma_{J(n)}^k)- \tfrac1{q_n}{\bm E_{q_n}(\tau_{I(n)}^i\times\Gamma_{J(n)}^j\times\Gamma_{J(n)}^k)} \right].\label{eq:Gamma2}
\end{align}
We first consider the last triple sum. By \cref{thm:scatterlessDetection} the logarithms in all summands are bounded uniformly in $n$, and
\begin{equation}\label{eq:GammaConvEffectiveIntervals1}
\sum_{i=1}^{N_{I(n)}}\sum_{j,k=1}^{M_{J(n)}}\absNorm{A\density^\dagger(\tau^i_{I(n)}\times\Gamma^j_{J(n)}\times\Gamma^k_{J(n)})-\frac{1}{q_n}\bm E_{q_n}(\tau^i_{I(n)}\times\Gamma^j_{J(n)}\times\Gamma^k_{J(n)})}
\end{equation} 
converges to zero by \cref{thm:stochConv} (with $r_n=1$) because $N_{I(n)}M^2_{J(n)}/q_n\le(N_{I(n)}/q_n)^{1/2}\to 0$.
It remains to show that the first multi sum in \eqref{eq:Gamma2} converges to zero. To this end, we prove that
\begin{equation}\label{eqn:logDifference}
	\Log{\frac{\forwardOpPN{\densityN{n}}(\tau\times\Gamma\times\Gamma')}{\mathcal{L}^1(\tau)\hd^2(\Gamma)\hd^2(\Gamma')}} - \Log{\frac{\forwardOpPN{\densityN{n}}(\tau^i_{I(n)}\times\Gamma_{J(n)}^j\times\Gamma_{J(n)}^k)}{\mathcal{L}^1(\tau^i_{I(n)})\hd^2(\Gamma_{J(n)}^j)\hd^2(\Gamma_{J(n)}^k)}}
\end{equation}
converges to zero uniformly in spacetime such that the multi sum goes to zero by (almost sure) boundedness of the measures $\bm E_{q_n}/q_n$ and $\forwardOp\density^\dagger$. We get (the scatter contributions cancel each other)
\begin{multline}
	\absNorm{\frac{\forwardOpPN{\densityN{n}}(\tau\times\Gamma\times\Gamma')}{\mathcal{L}^1(\tau)\hd^2(\Gamma)\hd^2(\Gamma')} - \frac{\forwardOpPN{\densityN{n}}(\tau^i_{I(n)}\times\Gamma_{J(n)}^j\times\Gamma_{J(n)}^k)}{\mathcal{L}^1(\tau^i_{I(n)})\hd^2(\Gamma_{J(n)}^j)\hd^2(\Gamma_{J(n)}^k)}}\\
	\le\absNorm{\frac{\forwardOpD{\densityN{n}}(\tau\times\Gamma\times\Gamma')}{\mathcal{L}^1(\tau)\hd^2(\Gamma)\hd^2(\Gamma')} - \frac{\forwardOpD{\densityN{n}}(\tau^i_{I(n)}\times\Gamma\times\Gamma')}{\mathcal{L}^1(\tau^i_{I(n)})\hd^2(\Gamma)\hd^2(\Gamma')}}\\
	+\absNorm{\frac{\forwardOpD{\densityN{n}}(\tau^i_{I(n)}\times\Gamma\times\Gamma')}{\mathcal{L}^1(\tau^i_{I(n)})\hd^2(\Gamma)\hd^2(\Gamma')} - \frac{\forwardOpD{\densityN{n}}(\tau^i_{I(n)}\times\Gamma_{J(n)}^j\times\Gamma_{J(n)}^k)}{\mathcal{L}^1(\tau^i_{I(n)})\hd^2(\Gamma_{J(n)}^j)\hd^2(\Gamma_{J(n)}^k)}}.\label{eq:Gamma3}
\end{multline}
Note that the first summand is zero whenever $I(n)=n$ due to $\tau=\tau^i_{I(n)}$ in that case, so assume $I(n)<n$.
Using the mean value theorem in time (which is allowed due to the temporal continuity by \cref{lemma:TimeLipschitzForwardOperator}),
\cref{lemma:TimeLipschitzForwardOperator} and \cref{AssumptionSetting}\ref{enm:timePartition} yield
\begin{multline}\label{eq:GammaConvForBetaPositiveTMP}
	\absNorm{\frac{\forwardOpD{\densityN{n}}(\tau\times\Gamma\times\Gamma')}{\mathcal{L}^1(\tau)\hd^2(\Gamma)\hd^2(\Gamma')} - \frac{\forwardOpD{\densityN{n}}(\tau_{I(n)}^i\times\Gamma\times\Gamma')}{\mathcal{L}^1(\tau^i_{I(n)})\hd^2(\Gamma)\hd^2(\Gamma')}}
	=\absNorm{\frac{\forwardOpD{\density_{n, t}}(\Gamma\times\Gamma')}{\hd^2(\Gamma)\hd^2(\Gamma')} - \frac{\forwardOpD{\density_{n,\hat t}}(\Gamma\times\Gamma')}{\hd^2(\Gamma)\hd^2(\Gamma')}}\\
	\lesssim\sqrt{\absNormS{t-\hat t}\min_\momentum S(\densityN{n},\momentum)}
	\lesssim\frac{1}{\sqrt{N_{I(n)}\beta_n}}
	\leq\frac{1}{\sqrt{(\beta_{n}q_{n-1})^{1-\varepsilon}}}
\end{multline}
for some $t\in\tau\subset\tau^i_{I(n)}$, $\hat t\in\tau^i_{I(n)}$, using the definition of $I(n)$ in the last inequality.
Thus, as $n\to\infty$, \eqref{eq:GammaConvForBetaPositiveTMP} converges to zero uniformly in all time intervals and detectors.
The second term in \eqref{eq:Gamma3} vanishes for $J(n)=n$, so assume $J(n)<n$.
We apply \cref{thm:scatterlessDetection}\ref{enm:ForwardDensitySplitting} and then again the mean value theorem,
this time on the detector pairs (which is admissible since any detector pair is path connected by \cref{AssumptionSetting}), to obtain
\begin{multline}\label{eqn:spatialEstimate}
	\absNorm{\frac{\forwardOpD{\densityN{n}}(\tau^i_{I(n)}\times\Gamma\times\Gamma')}{\mathcal{L}^1(\tau^i_{I(n)})\hd^2(\Gamma)\hd^2(\Gamma')} - \frac{\forwardOpD{\densityN{n}}(\tau^i_{I(n)}\times\Gamma_{J(n)}^j\times\Gamma_{J(n)}^k)}{\mathcal{L}^1(\tau^i_{I(n)})\hd^2(\Gamma_{J(n)}^j)\hd^2(\Gamma_{J(n)}^k)}}\\
	\leq\frac{1}{\mathcal{L}^1(\tau^i_{I(n)})}\int_{\tau^i_{I(n)}}\big|g(a_t, b_t)P[G*\density_{n,t}](\theta(a_t, b_t), s(a_t, b_t))\\
		-g(\hat a_t, \hat b_t)P[G*\density_{n,t}](\theta(\hat a_t, \hat b_t), s(\hat a_t, \hat b_t))\big|\dint t
\end{multline}
for some $(a_t, b_t)\in\Gamma\times\Gamma'\subset\Gamma_{J(n)}^j\times\Gamma_{J(n)}^k$ and $(\hat a_t, \hat b_t)\in\Gamma_{J(n)}^j\times\Gamma_{J(n)}^k$. Now, $(a,b)\mapsto g(a,b)P[G*\density_{n,t}](\theta(a,b),s(a,b))$ is Lipschitz
by smootheness of $g$, Lipschitz continuity of $G$, smoothness of $\theta$ and $s$ (on the set of interest $a-b$ is bounded away from zero), as well as boundedness of $\norm{\density_{n,t}}$ for a.e.\ $t$.
Thus, up to a constant factor, \eqref{eqn:spatialEstimate} is bounded by $\max\{\mathrm{diam}(\Gamma_{J(n)}^j),\mathrm{diam}(\Gamma_{J(n)}^k)\}\lesssim1/\sqrt{M_{J(n)}}$ (using \cref{AssumptionSetting}),
which converges to zero by definition of $J(n)$ (recall that we are in the setting $J(n)\neq n$ and that $\frac{q_n}{N_{I(n)}}\to\infty$).
Overall, \eqref{eq:Gamma3} uniformly converges to zero, and so does \eqref{eqn:logDifference}, as desired, since the arguments of the logarithms are uniformly bounded and uniformly bounded away from zero by \cref{thm:scatterlessDetection}.
Therefore, the $\liminf$-inequality also holds for $\density\neq0$.
\paragraph{limsup inequality.}
Let $\density\neq0$ (otherwise $\mathcal{E}_\infty(\density)=\infty$ and the limsup inequality is trivially fulfilled). We invoke \cref{thm:RecoverySequence}\ref{enm:RecoverySequenceB} and \cref{rmk:SupConvergenceLipschitzFunctions} to get measures $(\densityN{n},\momentumN{n})$ such that $\densityN{n}\convStar\density$ and $\BBEnergy(\densityN{n},\momentum_n)\lesssim\beta_n^{\delta-1}$ for any $\delta>0$ close to zero
and such that
\begin{align*}
\sup_x\absNorm{\int_D G(x-y)\dint\densityTN{t_n}{n}(y)-\int_D G(x-y)\dint\densityT{t}(y)}\convN 0.
\end{align*}
This directly gives us $\beta_n\min_\eta\BBEnergy(\densityN{n},\eta)\lesssim\beta_n^\delta\to0$ as well as $\norm{\forward{\densityN{n}}}\rightarrow\norm{\forward{\density}}$.
We can also deduce $\forwardBinnedPN{\densityN{n}}{n}(t,x,y)\to\forwardBinnedLimit{\density}(t,x,y)$ for almost every $(t,x,y)\in[0,T]\times\boundDomDeltaSq$:
For the scattering part this follows immediately from $\densityN{n}\convStar\density$. For the detection part we show the argument for case \ref{enm:TauNonconstGammaNonconst} and comment afterwards on the simple necessary modifications in the other cases.
Consider sequences $i_n,j_n,k_n$ with $(t,x,y)\in(\tau_n^{i_n}\times\DetectorPair[n]{j_n}{k_n})$, then the detection part is given by
\begin{equation}
\frac{\forwardD{\densityN{n}}(\tau_n^{i_n}\times\DetectorPair[n]{j_n}{k_n})}{\mathcal{L}^1(\tau^i_n)\hd^2(\Gamma_n^{j_n})\hd^2(\Gamma_n^{k_n})}
\le\frac{\absNorm{\forwardD{\densityTN{t_n}{n}}(\DetectorPair[n]{j_n}{k_n})-\forwardD{\densityT{t}}(\DetectorPair[n]{j_n}{k_n})}}{\hd^2(\Gamma_n^{j_n})\hd^2(\Gamma_n^{k_n})}+\frac{\forwardD{\densityT{t}}(\DetectorPair[n]{j_n}{k_n})}{\hd^2(\Gamma_n^{j_n})\hd^2(\Gamma_n^{k_n})}\label{eqn:LimSup2}
\end{equation}
for a sequence $t_n\to t$ with $\absNorm{t_n-t}\lesssim\frac{1}{N_n}$ obtained from the mean value theorem (exploiting the temporal continuity from \cref{lemma:TimeLipschitzForwardOperator}).
Denoting by $\omega$ the modulus of continuity of $\RNderivative{\forwardD{\densityT{t}}}{\hdhd}$ at $(x,y)$,
the last term deviates from $\RNderivative{\forwardD{\densityT{t}}}{\hdhd}(x,y)$
by at most $$\int_{\DetectorPair[n]{j_n}{k_n}}\omega(\absNorm{(a,b)-(x,y)})\dint(a,b)/(\hd^2(\Gamma_n^{j_n})\hd^2(\Gamma_n^{k_n}))\leq\omega(\max\{\mathrm{diam}(\Gamma_n^{j_n}),\mathrm{diam}(\Gamma_n^{k_n})\}),$$
which converges to zero as $n\to\infty$
(recall that by \cref{thm:scatterlessDetection}\ref{enm:ForwardDensitySplitting}
$\RNderivative{\forwardD{\densityT{t}}}{\hdhd}$ is continuous).
For the first term we use Lipschitz continuity of $G$ and \cref{thm:scatterlessDetection}\ref{enm:ForwardDensitySplitting} to obtain $\RNderivative{\forwardD{\densityT{t}}}{\hdhd}(x,y)=P[G*\densityT{t}](\theta(x,y), s(x,y))g(x,y)$ and get
\begin{align*}
&\tfrac{1}{\hd^2(\Gamma_n^{j_n})\hd^2(\Gamma_n^{k_n})}\absNorm{\forwardD{\densityTN{t_n}{n}}(\DetectorPair[n]{j_n}{k_n})-\forwardD{\densityT{t}}(\DetectorPair[n]{j_n}{k_n})}\\
\le&\tfrac{1}{\hd^2(\Gamma_n^{j_n})\hd^2(\Gamma_n^{k_n})}\int_{\DetectorPair[n]{j_n}{k_n}}\absNorm{\forwardD{\densityTN{t_n}{n}}-\forwardD{\densityT{t}}}\dint\hdhd\\
=&\tfrac{1}{\hd^2(\Gamma_n^{j_n})\hd^2(\Gamma_n^{k_n})}\int_{\DetectorPair[n]{j_n}{k_n}}\!\!\absNorm{g(x,y)}\!\Bigg\lvert\int_\R\int_\dom \!\!G(s(x,y)\!+\!l\theta(x,y)\!-\!z)\dint(\densityTN{t_n}{n}\!\!\!-\!\densityT{t})(z)\dint l
\Bigg\rvert\dint\hd^2\!\otimes\!\hd^2\!(x,y)\\
\lesssim&\sup_a\absNorm{\int_\dom G(a-z)\dint\densityTN{t_n}{n}(z)-\int_\dom G(a-z)\dint\densityT{t}(z)}\convN 0.
\end{align*}
As for the other settings,
in \ref{enm:TauConstGammaConst}, the indices $i_n,j_n,k_n$ are independent of $n$ and convergence of the left-hand side in \eqref{eqn:LimSup2} directly follows from the weak-* convergence of $\densityN{n}$;
in setting \ref{enm:TauNonconstGammaConst}, the indices $j_n=j,k_n=k$ are independent of $n$ so that the last term of \eqref{eqn:LimSup2} converges to $\frac{\forwardD{\densityT{t}}(\DetectorPair{j}{k})}{\hd^2(\Gamma^{j})\hd^2(\Gamma^{k})}$;
in setting \ref{enm:TauConstGammaNonconst}, finally, the indices $i_n$ are independent of $n$ and we simply replace
$\densityT{t}$ by $\int_{\tau^i}\densityT{t}\dint t/\mathcal{L}^1(\tau^i)$ and
$\densityTN{t_n}{n}$ by $\int_{\tau^i}\densityTN{t}{n}\dint t/\mathcal{L}^1(\tau^i)$ in the above argument.
Hence, overall we obtained $\forwardBinnedPN{\densityN{n}}{n}(t,x,y)\to\forwardBinnedLimit{\density}(t,x,y)$.
Now applying Fatou's lemma to \eqref{eq:GammaEffectiveQuantities} and the fact that the term in parentheses vanishes in the limit (as proven in the liminf inequality), we obtain
\begin{equation*}
\limsup_n-\frac{1}{q_n}\int\Log{\forwardBinnedPN{\densityN{n}}{n}}\dint \bm E_{q_n}
=\limsup_n-\int\Log{\forwardBinnedPN{\densityN{n}}{n}}\dint\forward{\density^\dagger}
\leq-\int\Log{\forwardBinnedLimit{\density}}\dint\forward{\density^\dagger}.
\end{equation*}
Altogether it follows
$\limsup_n\mathcal{E}^{\bm E_{q_n}}_n(\density_{n})\le\mathcal{E}_\infty(\density)$
as desired.
%
\end{proof}

\begin{remark}(Improved convergence conditions)\label{rmk:ImprovedConvergenceConditions}
We could have estimated \eqref{eq:Gamma0} by
\begin{multline*}
\absNorm{\int\Log{\forwardBinnedPN{\densityN{n}}{n}}\left(\RNderivative{\forward{\density^\dagger}}{\nu}\dint\HausdorffMeasSq\dint t - \dint \frac{\bm E_{q_n}}{q_n}\right)}\\
\le\norm{\Log{\forwardBinnedPN{\densityN{n}}{n}}}_\infty\sum_{i=1}^{N_n}\sum_{j,k=1}^{M_n}\absNorm{\frac{1}{q_n}\bm E_{q_n}(\tau_n^i\times\DetectorPair[n]{j}{k})-\forward{\density^\dagger}(\tau_n^i\times\DetectorPair[n]{j}{k})},
\end{multline*}
which would converge to zero by \cref{thm:stochConv} if $N_n M_n^2\in o(q_n)$ (recall that by \cref{thm:scatterlessDetection}\ref{enm:BoundednessForwardOp} the factor $\norm{\Log{\forwardBinnedPN{\densityN{n}}{n}}}_\infty$ is uniformly bounded in $n$ for $\densityN{n}\weakstarto\density\neq0$).
However, we could get rid of any condition on the detector number $M_n$
by exploiting the spatial regularity induced by the forward operator,
which allowed us to pass to artificially enlarged detectors $\Gamma_{J(n)}^k$.
Similarly, we could trade in conditions on the time interval number $N_n$ for conditions on the regularization strength $\beta_n$, which induce temporal regularity.

Said differently, the forward operator is smoothing in space and hence turns weak convergence into strong convergence along the spatial dimensions.
At the same time the measurement converges weakly in space, resulting in a ``weak times strong'' structure
so that no conditions on the (spatial) detectors are necessary for $\Gamma$-convergence.
However, the forward operator is not smoothing in time, resulting in an insufficient ``weak times weak'' structure in time
(since the Kullback--Leibler divergence is not jointly convex, it is not lower semi-continuous w.r.t.\ simultaneous weak convergence of the measurement and the reconstruction).
Therefore, additional smoothing in time becomes necessary,
either on the reconstruction $\density$ (via the Benamou--Brenier regularization)
or on the measurement $\bm E_{q_n}$ (by binning multiple coincidences into time intervals),
which results in the alternative conditions on regularization strength $\beta_n$ or time interval number $N_n$.
\end{remark}

\begin{lemma}[Equicoercivity of energies]\label{thm:EquiCoercivity}
Let $q_n\to\infty$ and $\UF_n\to\UF>0$ as $n\to\infty$ for positive sequences $q_n,\UF_n,\beta_n$,
and let $\bm E_{q_n}$ be a PPP with intensity measure $q_n\forward{\density^\dagger}>0$.
Then almost surely the sequence of functionals $\mathcal{E}^{\bm E_{q_n}}_n$ is equicoercive,
i.e.\ $\mathcal{E}^{\bm E_{q_n}}_n(\density)\geq\frac{\norm{\density}}\kappa-\kappa$ for some $\kappa>0$ and $n\ge\bm n$, where $\bm n$ is a suitable almost surely integer valued stopping time.
\end{lemma}

\begin{proof}
It is sufficient to consider $\density$ such that $\min_\momentum \BBEnergy(\density,\momentum)<\infty$ because otherwise there is nothing to show. By $\norm{\forwardOp\density}\geq\norm{\forwardOp^\sct\density}\gtrsim\norm{\density}$ and \cref{thm:scatterlessDetection} there exists some $C>0$ such that
\begin{multline*}
\mathcal{E}^{\bm E_{q_n}}_n(\density)
\geq\frac{\norm{\density}}C-\frac1{q_n}\int\log(C\norm{\density})\dint\bm E_{q_n}
=\tfrac{\norm{\density}}C-\log(C\norm{\density})\norm{\tfrac{\bm E_{q_n}}{q_n}}\\
\geq\norm{\tfrac{\bm E_{q_n}}{q_n}}\left(1-\log\left(2C^2\norm{\tfrac{\bm E_{q_n}}{q_n}}\right)\right)+\tfrac{1}{2C}\norm{\density}\ge -2C^2\norm{\tfrac{\bm E_{q_n}}{q_n}}^2+\tfrac{1}{2C}\norm{\density}.
\end{multline*}
Due to \cref{thm:flatConvergence} we have $\norm{\bm E_{q_n}/{q_n}}\to\norm{\forwardOp\density^\dagger}$ almost surely
so that for $\kappa$ large enough and $n$ large enough (the latter depending on the realization) the right-hand side is bounded below by $\frac{\norm{\density}}\kappa-\kappa$, as desired. The lower bound holds as soon as $\normS{\bm E_{q_n}/q_n}<\sqrt{\kappa}/\sqrt{2}C$, i.e.\ for $n\ge\bm n=\sup\{n\in\N \ | \  \normS{\bm E_{q_n}/q_n}\ge\sqrt{\kappa}/\sqrt{2}C \}$. Since
\begin{align*}
\{\bm n=n\}=\left\{ \normS{\bm E_{q_n}/q_n}\ge\sqrt{\kappa}/\sqrt{2}C \right\}\cap\bigcap_{m=n+1}^\infty\left\{ \normS{\bm E_{q_m}/q_m}<\sqrt{\kappa}/\sqrt{2}C \right\}
\end{align*}
is measurable, $\bm n$ is a stopping time. By convergence of $\normS{\bm E_{q_m}/q_m}$ it is almost surely integer valued.
\end{proof}

\begin{corollary}[Convergence of minimizers]\label{thm:ConvMinimizers}
In the setting of \cref{thm:GammaConv2} or \cref{thm:GammaConv}, let $\densityN{n}$ be a sequence of minimizers of $\mathcal{E}^{\bm E_{q_n}(\omega)}_n$.
For almost every $\omega\in\Omega$ (i.e.\ almost surely) it has a weakly-* converging subsequence, and any weak-* limit point $\density$ minimizes $\mathcal{E}_\infty$ with
\begin{align*}
\mathcal{E}_\infty(\density)
=\min\mathcal{E}_\infty
=\lim_{n\to\infty}\mathcal{E}^{\bm E_{q_n}(\omega)}_n(\densityN{n})
=\lim_{n\to\infty}\min\mathcal{E}^{\bm E_{q_n}(\omega)}_n.
\end{align*}
\end{corollary}
\begin{proof}
This standard implication follows from \cite[Cor.\,7.20, Thm.\,7.8]{DalMasoGammaConv} and \cref{thm:EquiCoercivity}.
\end{proof}

In case \ref{enm:TauNonconstGammaNonconst} any source of noise (time and space discretization, Poisson measurement noise) vanishes in the limit. Hence, for $\beta_n\to0$, we expect the sequence of minimizers of $\mathcal{E}^{\bm E_{q_n}(\omega)}_n$ to converge to the ground truth material distribution $\density^\dagger$ that generates the (Poisson) measurements. This is indeed the case: $\density^\dagger$ is the unique minimizer of $\mathcal E_\infty$ as we will show in the remainder of the section.

\begin{lemma}[Injectivity of forward operator]\label{thm:InjectivityForwardOp}
The map 
\begin{equation*}\textstyle
\measp(\dom)\ni\lambda\mapsto\RNderivative{\forwardD{\lambda}}{\HausdorffMeasSq}\in L^1(\boundDomDeltaSq)
\end{equation*} is injective.
\end{lemma}
\begin{proof}
	 \Cref{thm:scatterlessDetection}\ref{enm:ForwardDensitySplitting} implies that $\mu\mapsto\RNderivative{\forwardD{\mu}}{\HausdorffMeasSq}$ is injective if $\mu\mapsto P[G*\mu]$ is.
	 Due to the invertibility of the X-ray transform this is the case if and only if $\mu\mapsto G*\mu$ is injective.
	 Hence, let $\lambda,\mu\in\measp(\dom)$ with $G*\lambda=G*\mu$.
	 The convolution theorem for the Fourier transform $\mathcal{F}$ of compactly supported distributions \cite[Thm.\,7.1.15]{Ho90} thus yields
	 \begin{equation}\label{eqn:convThm}
	 \mathcal{F}(G)\mathcal{F}(\lambda)
	 =\mathcal{F}(G*\lambda)
	 =\mathcal{F}(G*\mu)
	 =\mathcal{F}(G)\mathcal{F}(\mu)
	 \end{equation}
	 on $\R^3$.
	 Since $G$ has compact support, $\mathcal{F}(G)$ is analytic by Schwartz's Paley--Wiener theorem \cite[Thm.\,7.23]{Rudin_FA} (even entire if extended to $\C^3$).
	 It is folklore that therefore its zero-level set $S=\{x\in\R^3\,|\,\mathcal{F}(G)(x)=0\}$ has measure zero (otherwise $G$ would be zero), see e.g.\ \cite{Mi20}.
	 Now \eqref{eqn:convThm} implies $\mathcal{F}(\lambda-\mu)=0$ on $\R^3\setminus S$.
	 Since $\lambda$ and $\mu$ also have compact support and therefore $\mathcal{F}(\lambda-\mu)$ is analytic,
	 $\mathcal{F}(\lambda-\mu)$ can only be zero on a nullset unless $\mathcal{F}(\lambda-\mu)=0$.
	 Therefore $\lambda=\mu$.
\end{proof}

\begin{corollary}[Convergence to ground truth]
In the setting of \cref{thm:GammaConv}, case \ref{enm:TauNonconstGammaNonconst}, $u=1$,
the minimizers of $\mathcal{E}^{\bm E_{q_n}}_n$ almost surely converge weakly-* to the ground truth $\density^\dagger\in\meas_c$.
\end{corollary}
\begin{proof}
By \cref{thm:InjectivityForwardOp}, $\forwardOp$ is injective and thus the functional
\begin{equation*}\textstyle
\mathcal{E}_\infty(\density)=\int\RNderivative{\forward{\density}}{\nu}-\RNderivative{\forward{\density^\dagger}}{\nu}\Log{\RNderivative{\forward{\density}}{\nu}}\dint\nu+\iota_{\meas_c}(\density)
\end{equation*}
is strictly convex.
One readily finds its unique minimizer to be $\density^\dagger$,
which by \cref{thm:ConvMinimizers} is almost surely the unique limit of any subsequence of minimizers of $\mathcal{E}^{\bm E_{q_n}}_n$.
\end{proof}



\bibliographystyle{unsrt}
\bibliography{bibliography}

\begin{thebibliography}{10}

\bibitem{PET_Base}
Marco Mauritz, Bernhard Schmitzer, and Benedikt Wirth.
\newblock A bayesian model for dynamic mass reconstruction from pet listmode
  data.
\newblock arXiv:2311.17784, 2023.

\bibitem{Schmitzer_DynamicCellImaging}
Bernhard Schmitzer, Klaus~P. Schäfers, and Benedikt Wirth.
\newblock Dynamic cell imaging in pet with optimal transport regularization.
\newblock {\em IEEE Transactions on Medical Imaging}, 39(5):1626--1635, 2020.

\bibitem{Bredies_Fanzon_dynamicIP}
Kristian Bredies and Silvio Fanzon.
\newblock An optimal transport approach for solving dynamic inverse problems in
  spaces of measures.
\newblock {\em ESAIM Math. Model. Numer. Anal.}, 54(6):2351--2382, 2020.

\bibitem{Bredies_Carioni_CondintionalGradientMethodOTRegularization}
Kristian Bredies, Marcello Carioni, Silvio Fanzon, and Francisco Romero.
\newblock A generalized conditional gradient method for dynamic inverse
  problems with optimal transport regularization.
\newblock {\em Found. Comput. Math.}, 23(3):833--898, 2023.

\bibitem{Candes-Fernandez-Granda}
Emmanuel~J. Cand\`es and Carlos Fernandez-Granda.
\newblock Towards a mathematical theory of super-resolution.
\newblock {\em Comm. Pure Appl. Math.}, 67(6):906--956, 2014.

\bibitem{Peyre-Duval-Denoyelle}
Quentin Denoyelle, Vincent Duval, and Gabriel Peyr\'{e}.
\newblock Support recovery for sparse super-resolution of positive measures.
\newblock {\em J. Fourier Anal. Appl.}, 23(5):1153--1194, 2017.

\bibitem{Holler_Schlueter_Wirth}
M.~Holler, A.~Schl\"{u}ter, and B.~Wirth.
\newblock Dimension reduction, exact recovery, and error estimates for sparse
  reconstruction in phase space.
\newblock {\em Appl. Comput. Harmon. Anal.}, 70:Paper No. 101631, 49, 2024.

\bibitem{PPPconcentrationInequality}
Patricia Reynaud-Bouret.
\newblock Adaptive estimation of the intensity of inhomogeneous {P}oisson
  processes via concentration inequalities.
\newblock {\em Probab. Theory Related Fields}, 126(1):103--153, 2003.

\bibitem{OTAppliedMath}
Filippo Santambrogio.
\newblock {\em Optimal transport for applied mathematicians}, volume~87 of {\em
  Progress in Nonlinear Differential Equations and their Applications}.
\newblock Birkh\"{a}user/Springer, Cham, 2015.
\newblock Calculus of variations, PDEs, and modeling.

\bibitem{bookPPP_last_penrose}
G\"{u}nter Last and Mathew Penrose.
\newblock {\em Lectures on the {P}oisson process}, volume~7 of {\em Institute
  of Mathematical Statistics Textbooks}.
\newblock Cambridge University Press, Cambridge, 2018.

\bibitem{bookPPP_kingman}
J.~F.~C. Kingman.
\newblock {\em Poisson processes}, volume~3 of {\em Oxford Studies in
  Probability}.
\newblock The Clarendon Press, Oxford University Press, New York, 1993.
\newblock Oxford Science Publications.

\bibitem{stochCoupling}
Hermann Thorisson.
\newblock {\em Coupling, stationarity, and regeneration}.
\newblock Probability and its Applications (New York). Springer-Verlag, New
  York, 2000.

\bibitem{stochCoupling2}
Benjamin Berkels and Benedikt Wirth.
\newblock Joint denoising and distortion correction of atomic scale scanning
  transmission electron microscopy images.
\newblock {\em Inverse Problems}, 33(9):095002, 41, 2017.

\bibitem{PoissonData}
Frank Werner and Thorsten Hohage.
\newblock Convergence rates in expectation for {T}ikhonov-type regularization
  of inverse problems with {P}oisson data.
\newblock {\em Inverse Problems}, 28(10):104004, 15, 2012.

\bibitem{KlenkeProbabilityTheory}
Achim Klenke.
\newblock {\em Probability theory}.
\newblock Universitext. Springer, London, second edition, 2014.
\newblock A comprehensive course.

\bibitem{DefinitionAssouadDimension}
Ville Suomala.
\newblock Intermediate value property for the {A}ssouad dimension of measures.
\newblock {\em Anal. Geom. Metr. Spaces}, 8(1):106--113, 2020.

\bibitem{chizat_unbalancedOT}
Lena{\"i}c Chizat.
\newblock {\em {Unbalanced Optimal Transport : Models, Numerical Methods,
  Applications}}.
\newblock Theses, {Universit{\'e} Paris sciences et lettres}, November 2017.

\bibitem{BenamouBrenier}
Jean-David Benamou and Yann Brenier.
\newblock A computational fluid mechanics solution to the {M}onge-{K}antorovich
  mass transfer problem.
\newblock {\em Numer. Math.}, 84(3):375--393, 2000.

\bibitem{Folland_RealAnalysis}
Gerald~B. Folland.
\newblock {\em Real analysis}.
\newblock Pure and Applied Mathematics (New York). John Wiley \& Sons, Inc.,
  New York, second edition, 1999.
\newblock Modern techniques and their applications, A Wiley-Interscience
  Publication.

\bibitem{DalMasoGammaConv}
Gianni Dal~Maso.
\newblock {\em An introduction to {$\Gamma$}-convergence}, volume~8 of {\em
  Progress in Nonlinear Differential Equations and their Applications}.
\newblock Birkh\"{a}user Boston, Inc., Boston, MA, 1993.

\bibitem{Ho90}
Lars H\"{o}rmander.
\newblock {\em The analysis of linear partial differential operators. {I}},
  volume 256 of {\em Grundlehren der mathematischen Wissenschaften [Fundamental
  Principles of Mathematical Sciences]}.
\newblock Springer-Verlag, Berlin, second edition, 1990.
\newblock Distribution theory and Fourier analysis.

\bibitem{Rudin_FA}
Walter Rudin.
\newblock {\em Functional analysis}.
\newblock International Series in Pure and Applied Mathematics. McGraw-Hill,
  Inc., New York, second edition, 1991.

\bibitem{Mi20}
B.~S. Mityagin.
\newblock The zero set of a real analytic function.
\newblock {\em Mat. Zametki}, 107(3):473--475, 2020.

\end{thebibliography}

\end{document}